\documentclass{aims}
\usepackage{amsmath}
\usepackage{amssymb, bm, bbm}
\usepackage{esint}
\usepackage[frak]{paper_diening}
\usepackage{paralist}
\usepackage{enumitem}
\usepackage{epsfig} 
\usepackage{epstopdf} 
\usepackage[colorlinks=true]{hyperref}

\newtheorem{theorem}{Theorem}[section]
\newtheorem{corollary}{Corollary}

\newtheorem{lemma}[theorem]{Lemma}
\newtheorem{proposition}{Proposition}

\theoremstyle{definition}

\newtheorem{remark}{Remark}

\newcommand{\E}{\mathcal{E}}
\newcommand{\W}{\mathcal{W}}
\newcommand{\R}{\ensuremath{\mathbb{R}}}

\renewcommand{\epsilon}{\varepsilon}

\renewcommand{\Xi}{X}

\newcommand{\Rn}{\mathbb{R}^n}

\renewcommand{\phi}{G}

\newcommand{\kappaA}{{\kappa_{\!\mathcal{A}}}}

\def\la1{\lambda_1}

\newcommand{\kabs}[1]{\ensuremath{\vert#1\vert}}

\numberwithin{equation}{section}



\makeatletter
\let\orgdescriptionlabel\descriptionlabel
\renewcommand*{\descriptionlabel}[1]{%
  \let\orglabel\label
  \let\label\@gobble
  \phantomsection
  \edef\@currentlabel{#1}%
  \let\label\orglabel
  \orgdescriptionlabel{#1}%
}
\makeatother

\allowdisplaybreaks[3] 
\begin{document}
\title{Partial regularity for steady double phase fluids }

\author[G. Scilla and B. Stroffolini]{}

\maketitle


\centerline{\scshape Giovanni Scilla and Bianca Stroffolini}
\medskip
{\footnotesize
 \centerline{Dipartimento di Matematica ed Applicazioni "R. Caccioppoli"}
   \centerline{Universit\`{a} degli Studi di Napoli Federico II}
 \centerline{Via Cintia, Monte Sant'Angelo 80126 Napoli}
   \centerline{Italy}
}

\bigskip

\begin{abstract}
We study partial H\"older regularity for nonlinear elliptic systems in divergence form with double-phase growth, modeling double-phase non-Newtonian fluids in the stationary case. 
\end{abstract}

\\
\begin{flushright}
{\sl{Dedicated to Rosario Mingione for his 50th birthday}}
\end{flushright}

\section{Introduction}

Our research was inspired by two main contributions of Rosario: the regularity for non-Newtonian fluids \cite{acerbimingione} and the double phase problems, \cite{ColomboMingione15, baronicolombomingione}.
\subsection{The problem and the main assumptions} This article deals with  nonlinear elliptic systems in divergence form, modeling double-phase non-Newtonian fluids in the stationary case:
\begin{equation}
{\rm div}\,\bfu=0\,,\quad{\rm div}\,{\bf a}(x,\E\bfu) + D \pi =\bfu[D \bfu] + f  
\,\, \mbox{ in $\Omega$.}
\label{eq:1.1ok}
\end{equation}
Here $\Omega\subseteq\R^n$ denotes an open, bounded set, $n\geq3$, the vector-valued
map $\bfu:\Omega\to\R^n$ can be interpreted as the stationary velocity field of a fluid, and the scalar function $\pi:\Omega\to\R$ plays the role of the pressure.
 The stress tensor will have double phase growth, as a function of the symmetric gradient $\E\bfu$.

 For the nonlinear diffusion term  
 ${\bf a}:\Omega\times\R_{\rm sym}^{n\times n}\to \R_{\rm sym}^{n\times n}$ we are considering a double phase growth condition of the type 
 \begin{equation}
H(x,t):= t^p + \mu(x) t^q\,,
\label{eq:H}
\end{equation}
 where \begin{equation}
2\leq p<q\leq p+\frac{\alpha p}{n}\,, \quad 0<\alpha\leq1\,,
\label{eq:pq}
\end{equation}
 and the modulating coefficient $\mu$ is non negative,  bounded and H\"older continuous with exponent $\alpha$. 

{Notice that in the region $\{\mu=0\}$, we have $H(x,t)=t^p$ so that $H$ has $p$-phase, while in the region $\{\mu>0\}$ the function $H$ has $q$-phase. Moreover, for each $x\in\Omega$ such that $\mu(x)>0$, the function $t\to H(x,t)$ is an $N$-function (see Section~\ref{sec:basicNfunctions}) complying with \eqref{(2.1celok)} where $g_1=p$ and $g_2=q$.}
The precise assumptions are:
\begin{equation}
\begin{cases}
|{\bf a}(x,\bm\xi)|+ |D_\xi \bfa(x,\bm\xi)|(1+|\bm\xi|) \leq L H'(x,1+|\bm\xi|) \\
\langle D_\xi {\bf a}(x,\bm\xi)\bm\lambda,\bm\lambda\rangle \geq \nu H''(x,1+|\bm\xi|)|\bm\lambda|^2 \,.
\end{cases}
\label{eq:1.8ok}
\end{equation}
for every $x\in\Omega$ and $\bm\xi\in\R_{\rm sym}^{n\times n}$, $\bm\lambda \in \R^{n\times n}$ and for some $0<\nu\leq L$, {where $H'(x,t)$ and $H''(x,t)$ denote the first and second derivative of $t\to H(x,t)$, respectively.} 
Note that the second inequality in \eqref{eq:1.8ok} implies 
\begin{equation}
({\bf a}(x,\bm\xi_1) - {\bf a}(x,\bm\xi_2)) : (\bm\xi_1-\bm\xi_2) \geq \tilde{\nu} H''(x,1+|\bm\xi_1|+|\bm\xi_2|)|\bm\xi_1-\bm\xi_2|^2
\label{eq:1.9ok}
\end{equation}
for every $x\in\Omega$ and $\bm\xi_1, \bm\xi_2\in \R_{\rm sym}^{n\times n}$.
We further assume the existence of a nondecreasing and concave function $\omega:[0,\infty)\to[0,1]$ with $\omega(0)=0$ such that
\begin{equation}
\begin{split}
|{\bf a}(x_1,\bm\xi)-{\bf a}(x_2,\bm\xi)| \leq & L |\mu(x_1)-\mu(x_2)| 
(1+|\bm\xi|)^{q-1}\,,
\end{split}
\label{eq:1.11ok}
\end{equation}
and
\begin{equation}
|D_\xi{\bf a}(x,\bm\xi_1)-D_\xi{\bf a}(x,\bm\xi_2)| \leq L \omega\left(\frac{|\bm\xi_1-\bm\xi_2|}{1+|\bm\xi_1|+|\bm\xi_2|}\right) H''(x,1+|\bm\xi_1|+|\bm\xi_2|)\,,
\label{eq:1.12ok}
\end{equation}
for every $x, x_1, x_2\in\Omega$ and $\bm\xi, \bm\xi_1, \bm\xi_2\in \R_{\rm sym}^{n\times n}$.

For the force term $f$ we require that
\begin{equation}
f\in L^{n(1+\beta)}_{\rm loc}(\Omega;\R^n)
\label{eq:1.6bog}
\end{equation}
for some $\beta>0$.

\subsection{The main result}
{A pair $(\bfu, \pi)\in W^{1,1}(\Omega;\R^n)\times W^{1,1}(\Omega, \R)$ with $H(\cdot,|\E\bfu|)\in L^1_{\rm{loc}}(\Omega)$ is a \emph{weak solution} to \eqref{eq:1.1ok} if and only if ${\rm div}\,\bfu=0$ in $\Omega$ in the sense of distributions and
\begin{equation}
\int_\Omega [\langle{\bf a}(x,\E\bfu), \E\bm\varphi\rangle + \pi\,{\rm div}\,\bm\varphi]\,\mathrm{d}x = \int_\Omega [\bfu[D \bfu] + f]\cdot\bm\varphi\,\mathrm{d}x
\label{eq:system1}
\end{equation}
holds for all $\bm\varphi\in C_0^\infty(\Omega,\R^n)$. If we test with divergence free vector fields
\begin{equation*}
\bm\varphi\in C_{0,{\rm div}}^\infty(\Omega,\R^n):= \{\bm\psi\in C_0^\infty(\Omega,\R^n):\,\, {\rm div}\,\bm\psi=0\}\,,
\end{equation*}
the pressure term in \eqref{eq:system1} vanishes and the system reduces to ${\rm div}\,\bfu=0$ in $\Omega$ in the sense of distributions and
\begin{equation}
\int_\Omega \langle{\bf a}(x,\E\bfu), \E\bm\varphi\rangle \,\mathrm{d}x = \int_\Omega [\bfu[D \bfu] + f]\cdot\bm\varphi\,\mathrm{d}x
\label{eq:system2}
\end{equation}
whenever $\bm\varphi\in C_{0,{\rm div}}^\infty(\Omega,\R^n)$. Within this setting, we say that $\bfu \in W^{1,1}(\Omega;\R^n)$ with $H(\cdot,|\E\bfu|)\in L^1_{\rm{loc}}(\Omega)$ is a {weak solution} to \eqref{eq:1.1ok} if and only if ${\rm div}\,\bfu=0$ in $\Omega$ in the sense of distributions and \eqref{eq:system2} holds.}

The main result of the paper is a partial Morrey regularity result:

\begin{theorem}\label{thm:thm1.1ok}
Let $H:\Omega\times[0,+\infty)\to[0,+\infty)$ be defined as in \eqref{eq:H}, with 
\eqref{eq:pq}. Assume that the vector field ${\bf a}:\Omega\times \R_{\rm sym}^{n\times n}\to \R_{\rm sym}^{n\times n}$ complies with \eqref{eq:1.8ok}, \eqref{eq:1.11ok}, \eqref{eq:1.12ok}, and that 
\eqref{eq:1.6bog} holds. Let $(\bfu, \pi)\in W^{1,1}(\Omega;\R^n)\times W^{1,1}(\Omega, \R)$ with $H(\cdot,|\E\bfu|)\in L^1_{\rm{loc}}(\Omega)$ be a weak solution to \eqref{eq:1.1ok}.
 Then there exists an open subset $\Omega_0 \subset \Omega$ such that
\begin{equation*}
 \bfu \in C^{0, \beta}_{\rm{loc}}\left(\Omega_0 ,\R^{n}\right) , \qquad \kabs{\Omega \setminus \Omega_0} \, = \, 0
\end{equation*}
for every $\beta\in (0,1)$. Moreover, $\Omega \setminus \Omega_0\subset \Sigma_1\cup\Sigma_2$ where
\begin{equation*}
\begin{split}
&\Sigma_1:=\left\{x_0\in\Omega:\,\, \mathop{\lim\inf}_{\varrho\searrow 0}\dashint_{B_\varrho(x_0)}|D\bfu-(D\bfu)_{x_0,\varrho}|\,\mathrm{d}x>0\right\}\,,\\
&\Sigma_2:=\left\{x_0\in\Omega:\,\, \mathop{\lim\sup}_{\varrho\searrow 0}\left[\dashint_{B_\varrho(x_0)}H(x,1+|\E\bfu|)\,\mathrm{d}x + (|D\bfu|^p)_{x_0,\varrho} + |(\bfu)_{x_0,\varrho}|\right]=+\infty\right\}\,.
\end{split}
\end{equation*}
\end{theorem}

\subsection{Historical background} The study of PDEs/functionals with $(p,q)$ growth, corresponding to the case $\mu(x) \equiv1$, was initiated by Marcellini in the 90's \cite{marcellini91} and has been generalized in many directions, see \cite{espoleomin, mingionedark,carozzakripass,mingioneradulescu, cristianaleo, marcellini22} and references therein. The main feature of these problems is a gap between coercitivity and growth condition, that can eventually lead even to unbounded minimizers/solutions.
In order to prevent such irregular behaviour, one has to impose a bound on the ratio $\frac {q}{ p} <1+\frac1{n}$.
This bound has been improved recently by Bella and Sch\"affner to $\frac {q}{ p} <1+\frac2{n-1}$, \cite{bellaschaffner}. More recently, in \cite{cristiana22, cristianame}, some tools from nonlinear potential theory have been employed to infer sharp partial regularity results for relaxed minimizers of degenerate/singular, nonuniformly elliptic quasiconvex functionals of $(p,q)$-growth.
\par
As for the double phase problems, the ellipticity could eventually change from $p$-phase to $q$-phase, since the modulating coefficient could annihilate. Pioneer were Colombo and Mingione, Baroni, Colombo and Mingione  who established Morrey estimates for local minimizers under the condition $\frac {q}{ p} <1+\frac{\alpha}{n}$, where $\alpha$ is the modulus of continuity of the modulating coefficient, and also they considered borderline cases (of logarithmic type) and a unified approach to variable exponent, \cite{ColomboMingione15}, \cite{baronicolombomingione}. Their method relies on a refined alternative between $p$-phase and $(p,q)$-phase at every scale combined with an exit time argument.
Their intrinsic approach was pushed further by Ok, who was able to find 
an easier proof 
in the superquadratic case, \cite{OKJFA18} and \cite{OKNLA18}. \par

In the literature, nonlinear elliptic systems of \eqref{eq:1.1ok} when  the symmetric gradient has $p(x)$ growth are known as electrorheological fluids. The model was introduced by 
Rajagopal and \Ruzicka, \cite{rajruz01}, while the existence theory was studied by \Ruzicka, \cite{ruzickabook}. Acerbi and Mingione, \cite{acerbimingione}, first established  regularity for solutions of such systems. They proved that the gradient is partially H\"older continuous on a set of full measure under suitable continuity assumptions on the exponent $p(x)$.\par

Our aim is to study regularity properties of solutions of the system  \eqref{eq:1.1ok} where the symmetric gradient presents a double phase growth. Featuring the double phase case considered by Colombo and Mingione, we cannot expect more than a partial Morrey regularity result.
It is worth mentioning the paper \cite{BDHS 12}  where the authors considered electrorheological fluids with discontinuous coefficients and proved Morrey partial regularity.\par
Recently, the uniqueness of small solutions for steady double phase fluids has been investigated in \cite{ABC} for $\frac65<p<2<q.$

\subsection{Strategy of the proof}

We briefly explain the strategy of the proof of the main result. First, we stress the fact that our argument relies on two technical major ingredients which, to the best of our knowledge, appear for the first time in this paper. 

The first one is the \emph{$\mathcal{A}$-Stokes approximation lemma} for $\mathcal{A}$-Stokes systems with general growth (Theorem~\ref{AStokes}), which extends an analogous result for Stokes systems with $p$-growth provided by \cite[Theorem~4.2]{BrDieFuchs12}. The second one, contained in Lemma~\ref{lem:higint}, is a \emph{higher integrability result} for weak solutions to double phase Stokes systems \eqref{eq:1.1ok} with lower order terms. Namely, we can find an exponent $s>1$ depending on the data of the problem, such that
\begin{equation*}
\left(\dashint_{B_{r}(x_0)}[H(x,1+|\E\bfu|)]^{s}\,\mathrm{d}x\right)^{\frac{1}{s}} \leq c \dashint_{B_{2r}(x_0)}H(x,1+|\E\bfu|)\,\mathrm{d}x + c\dashint_{B_{2r}(x_0)}(r^\frac{p}{2}|D\bfu|^p+|\bfu|^{p^*}+1) \, \mathrm{d}x\,.
\end{equation*}
Analogous results were obtained for Stokes systems with $p(\cdot)$-growth, see \cite[Theorem~4.2]{acerbimingione} and \cite[Lemma~2.10]{BDHS 12}.

When addressing the problem of partial regularity, the leading quantity is the Campanato-type excess functional
\begin{equation*}
\Phi(x_0,\varrho):=\dashint_{B_\varrho(x_0)} H_{1+|(\E\bfu)_{x_0,\varrho}|} (x_0, |\E\bfu-(\E\bfu)_{x_0,\varrho}|)\,\mathrm{d}x
\end{equation*}
(see \eqref{eq:excess1}). 
We are dealing with the non-degenerate setting; i.e., when
\begin{equation}
\Phi(x_0,\varrho)\leq H(x_0,1+|(\E\bfu)_{x_0,\varrho}|)\,,
\label{(eqND)}
\end{equation}
so that we can linearize the problem, via the $\mathcal{A}$-Stokes approximation, Theorem~\ref{AStokes}.  
The linearization procedure of Lemma~\ref{lem:lemma4.2ok} provides a suitable rescaling ${\bf w}$ of the solution $\bfu$ which is shown to be approximately $\mathcal{A}$-Stokes on a ball $B_\varrho(x_0)$, and the deviation from being $\mathcal{A}$-Stokes is measured in terms of the ``hybrid'' excess
\begin{equation}
\Phi(x_0,\varrho) + \varrho^{\widetilde{\gamma}}H(x_0, 1+|(\E\bfu)_{x_0,\varrho}|)\,.
\label{eq:hybridexcess}
\end{equation}
Then, if \eqref{eq:hybridexcess} is small enough, Theorem~\ref{AStokes} applies and ensures the existence of a smooth $\mathcal{A}$-Stokes function $\bf h$ which is close in an integral sense to $\bf w$ and for which classical decay estimates are available, see \cite{fuchseregin00}. 
Scaling back from $\bf w$ to $\bfu$, this allows to prove an excess-decay estimate, which, in turn, permits the iteration of the rescaled excess $\frac{\Phi(x_0,\varrho)}{H(x_0,1+|(\E\bfu)_{x_0,\varrho}|)}$ and of a ``Morrey-type'' excess 
\begin{equation*}
\Theta(x_0,\varrho):=\varrho^\frac{1}{2} [H(x_0,\cdot)]^{-1}\left(\dashint_{B_{\varrho}}H(x_0,1+|D \bfu|)\,\mathrm{d}x\right)
\end{equation*}
at each scale. 

Namely,
there exists $\vartheta\in(0,1)$ such that, if the smallness conditions 
\begin{equation*}
\frac{\Phi(x_0,\varrho)}{H(x_0,1+|(\E\bfu)_{x_0,\varrho}|)}\leq \varepsilon_*\,,\qquad \Theta(x_0,\varrho)\leq\delta_*\,, \qquad |(\bfu)_{x_0,\varrho}|\leq\frac{1}{2}M
\end{equation*}
hold on some ball $B_\varrho(x_0)$, then 
\begin{equation*}
\frac{\Phi(x_0,\vartheta^m\varrho)}{H(x_0, 1+|(\E\bfu)_{x_0,\vartheta^m\varrho}|)}\leq \varepsilon_*\,,\qquad  \Theta(x_0,\vartheta^m\varrho)\leq\delta_*\,, \qquad |(\bfu)_{x_0,\vartheta^m\varrho}|\leq M
\end{equation*}
hold for every $m=0,1,\dots.$. 
Notice that, if on  one hand $|(D{\bf u})_{x_0,\varrho}|$ might blow up in the iteration since we cannot expect $C^1$-regularity; on the other hand, the Morrey excess $\Theta(x_0,\vartheta^k\varrho)$ stays bounded, exactly as it should be for a $C^{0,\alpha}$-regularity result.  The smallness of $\Theta$
at any level ensure H\"{o}lder continuity of $\bfu$ at $x_0$ provided the excess functionals $\Phi$ and $\Theta$ are small at some initial radius
$\varrho$ (actually, this holds in a neighborhood of $x_0$, since these smallness conditions are open). Finally, it is then proven that such a smallness condition on the excesses is indeed satisfied on the complement of the set $\Sigma_1\cup\Sigma_2$ of Theorem~\ref{thm:thm1.1ok}. 
{Notice that we were already using the Morrey type excess in \cite{goodscistro}, where we faced the partial regularity for discontinuous quasiconvex integrals with general growth and \cite{okscistro}  where  we considered the boundary partial regularity.}

\noindent
{\emph{Outline of the paper.} The paper is organized as follows. Section~\ref{sec:prelim} collects some basic definitions and auxiliary results useful throughout the paper. Specifically, Section~\ref{sec:basicNfunctions} contains basic facts about Orlicz functions, while Sections~\ref{sec:bogow}--\ref{sec:sobkorn} concern with typical auxiliary results in the study of systems depending on the symmetrized gradient, as a lemma of Bogowski\v{\i} and Sobolev-Korn-type inequalities. In Section~\ref{sec:Astokes} we formulate the $\mathcal{A}$-Stokes approximation lemma, which together with the self improving properties of weak solutions (Section~\ref{sec:selfimprove})}, is a key tool in the linearization procedure. The problem of the partial regularity of weak solutions starts with Section~\ref{sec:decayestimate}. Namely, in Section~\ref{sec:caccioppolitype} we prove the necessary Caccioppoli-type estimates, in Section~\ref{sec:almostAStokes} we show that a suitable scaling of the weak solution is almost $\mathcal{A}$-Stokes, while in Section~\ref{sec:excessdecay} we get some excess decay estimates. Finally, Section~\ref{sec:proofmainthm} is entirely devoted to the proof of the main result, Theorem~\ref{thm:thm1.1ok}.

\section{Preliminaries and auxiliary results} \label{sec:prelim}

\subsection{Basic notation}\label{sec:notation}
We denote by $\Omega$ an open bounded domain of $\R^n$. For $x_0\in\R^n$ and $r>0$, $B_r(x_0)$ is the open ball of radius $r$ centred at $x_0$. In the case $x_0=0$, we will often use the shorthand $B_r$ in place of $B_r(x_0)$. If $f\in L^1(B_r(x_0))$, we denote the average of $f$ by
\begin{equation*}
(f)_{x_0,r}:= \dashint_{B_r(x_0)} f\,\mathrm{d}x\,,
\end{equation*}
and we use the abbreviate notation $(f)_r$ for $(f)_{0,r}$. We denote by $\R^{n\times n}_{\rm sym}$ the set of all
symmetric $n\times n$ matrices. For $x,y\in\R^n$, we denote their tensor product by $x\otimes y:=\{x_iy_j\}_{i,j}\in \R^{n^2}$ ad their tensor symmetric product by $x\odot y:=\frac{1}{2}(x\otimes y + y\otimes x)\in \R^{n\times n}_{\rm sym}$. For a function $\bfu=(u^i)\in L^1(\Omega)$, we denote by $\E\bfu$ and $\mathcal{W}\bfu$ its symmetric and skew-symmetric distributional derivative, respectively:
\begin{equation*}
\E\bfu \equiv (\E\bfu)_{i,j}:=\frac{\partial_j u^i + \partial_i u^j}{2}\,, \quad \mathcal{W}\bfu \equiv (\mathcal{W}\bfu)_{i,j}:=\frac{\partial_j u^i - \partial_i u^j}{2} \,.
\end{equation*}
If $p>1$, then $p':=\frac{p}{p-1}$ denotes the conjugate exponent of $p$. If $1<p<n$, the number $p^*:=\frac{np}{n-p}$ stands for the Sobolev conjugate exponent of $p$, whereas $p^*$ is any real number if $p\geq n$.

\subsection{Some basic facts on $N$--functions} \label{sec:basicNfunctions}

We recall here some elementary definitions and basic results about Orlicz functions. The following definitions and results can be found, e.g., in \cite{Kras, Kufn, Bennett, Adams}. 


A real-valued function $G:[0,\infty)\to[0,\infty)$ is said to be an \emph{$N$-function} if it is convex and satisfies the following conditions: $\phi(0)=0$, $G$ admits the derivative $\phi'$ and this derivative is right continuous, non-decreasing and satisfies $\phi'(0) = 0$, $\phi'(t)>0$ for $t>0$, and $\lim_{t\to \infty} \phi'(t)=\infty$. 

We assume that $G$ also satisfies
\begin{equation}
g_1\leq \inf_{t>0}\frac{tG'(t)}{G(t)}\leq \sup_{t>0}\frac{tG'(t)}{G(t)}\leq g_2\,,
\label{(2.1celok)}
\end{equation}
for some $1<g_1\leq g_2 <\infty$. For instance, $G(t):=t^p$, $1<p<\infty$, is an $N$-function complying with \eqref{(2.1celok)} where $g_1=g_2=p$. Also $G(t):=H(x_0,t)$, where $H$ is defined in \eqref{eq:H}, is an $N$-function satisfying \eqref{(2.1celok)} with $g_1=p$ and $g_2=q$. Moreover, it can be seen that if $G\in C^2(0,\infty)$ and satisfies 
\begin{equation}
0<g_1-1\leq \inf_{t>0}\frac{tG''(t)}{G'(t)}\leq \sup_{t>0}\frac{tG''(t)}{G'(t)}\leq g_2-1\,,
\label{(2.1celokbis)}
\end{equation}
then $G$ is an $N$-function and \eqref{(2.1celok)} holds.

$\phi^\ast$ denotes the Young-Fenchel-Yosida conjugate function of $\phi$, given by $\phi^*(t)= \sup_{s \geq 0} (st - \phi(s))$. It is again an $N$-function; it satisfies \eqref{(2.1celok)} with $\frac{g_2}{g_2-1}$ and $\frac{g_1}{g_1-1}$ in place of $g_1$ and $g_2$, respectively. Also, $(\phi^\ast)^\ast = \phi$.

\begin{proposition}\label{prop:properties}
Let $G:[0,\infty)\to[0,\infty)$ be an $N$-function complying with \eqref{(2.1celok)}. Then 
\begin{description}
\item[(i)] the mappings
\begin{equation*}
t\in(0,+\infty)\to \frac{G'(t)}{t^{g_1-1}}\,,\,\, \frac{G(t)}{t^{g_1}} \mbox{ \,\, and \,\, } t\in(0,+\infty)\to \frac{G'(t)}{t^{g_2-1}}\,,\,\, \frac{G(t)}{t^{g_2}}
\end{equation*}
are increasing and decreasing, respectively. In particular,
\begin{equation}
\begin{split}
G(at) & \leq a^{g_1}G(t)\,,\,\,\,\,\,\, G'(at) \leq a^{g_1-1}G'(t)\,, \quad 0<a<1\,, \\
G(bt) & \leq b^{g_2}G(t)\,,\,\,\,\,\,\, G'(bt) \leq b^{g_2-1}G'(t)\,,\quad b>1\,.
\end{split}
\label{eq:2.2ok}
\end{equation}
Moreover,
\begin{equation}
G^*(at) \leq a^\frac{g_2}{g_2-1}G(t)\,,\quad G^*(bt) \leq a^\frac{g_1}{g_1-1}G(t) \,.
\label{eq:2.3ok}
\end{equation}
\item[(ii)] 
\begin{equation}
G(s+t) \leq 2^{g_2-1}(G(s)+G(t))\,;
\end{equation}
\item[(iii)] (Young's inequality) for any $\lambda\in(0,1]$ it holds that
\begin{equation}
\begin{split}
st & \leq \lambda^{-g_2+1}G(s) + \lambda G^*(t)\,, \\
st & \leq \lambda G(s) + \lambda^{-\frac{1}{g_1-1}} G^*(t)\,.
\end{split}
\label{eq:2.5ok}
\end{equation}
\item[(iv)] there exists a constant $c=c(g_1,g_2)>1$ such that
\begin{equation}
c^{-1}G(t) \leq G^*(t^{-1}G(t)) \leq c G(t)\,.
\label{eq:2.6ok}
\end{equation}
\end{description}
\end{proposition}

We say that $\phi$ satisfies the \emph{$\Delta_2$-condition} if there exists $c > 0$ such that for all $t \geq 0$ holds $\phi(2t) \leq c\,
\phi(t)$. We denote the smallest possible such constant by $\Delta_2(\phi)$. Since $\phi(t) \leq \phi(2t)$, the $\Delta_2$-condition is equivalent to $\phi(2t) \sim \phi(t)$, where ``$\sim$'' indicates the equivalence between $N$-functions.\par
%

In particular, from \eqref{eq:2.2ok} it follows that both $G$ and $G^*$  satisfy the $\Delta_2$-condition with constants $\Delta_2(\phi)$ and $\Delta_2(\phi^*)$ determined by $g_1$ and $g_2$. We will denote by $\Delta_2({\phi, \phi^\ast})$ constants depending on $\Delta_2(\phi)$ and $\Delta_2(\phi^*)$. Moreover, for $t>0$ we have
\begin{equation}
\phi(t) \sim \phi'(t)\,t\,, \qquad \phi(t) \sim \phi''(t)\,t^2\,.
\label{ineq:phiast_phi_p}
\end{equation}

The following inequalities hold for the inverse function $G^{-1}$:
\begin{align}
a^{\frac{1}{g_1}}G^{-1}(t)\leq &G^{-1}(at)\leq a^{\frac{1}{g_2}}G^{-1}(t)
\label{(2.3a)} 
\end{align}
for every $t\geq0$ with $0<a\leq1$. The same result holds also for $a\geq1$ by exchanging the role of $g_1$ and $g_2$.

%


Another important set of tools are the \emph{shifted} $N$-functions
$\{\phi_a \}_{a \ge 0}$ (see \cite{DIEETT08}). We define for $t\geq0$
\begin{equation}
  \label{eq:phi_shifted}
  \phi_a(t):= \int _0^t G_a'(s)\, \mathrm{d}s\qquad\text{with }\quad
  \phi'_a(t):=\phi'(a+t)\frac {t}{a+t}.
\end{equation}
We have the following relations:
\begin{align}
&\phi_a(t) \sim \phi'_a(t)\,t\,; \nonumber \\ 
&\phi_a(t) \sim \phi''(a+t)t^2\sim\frac{G(a+t)}{(a+t)^2}t^2\sim \frac{G'(a+t)}{a+t}t^2\,,\label{(2.6b)}\\
& \phi(a+t)\sim [\phi_a(t)+\phi(a)]\,,\label{(2.6c)} \\
& G(a+t)\leq cG_a(t) \,\, \mbox{for $t\geq a$. } \label{(2.6d)}
\end{align}
The families $\{\phi_a \}_{a \ge 0}$ and
$\{(\phi_a)^* \}_{a \ge 0}$ satisfy the $\Delta_2$-condition uniformly in $a \ge 0$. 

We recall also that, by virtue of \cite[Lemma~30]{DIEETT08}, uniformly in $\lambda\in[0,1]$ and $a\geq0$ holds
\begin{equation}
G_a^*(\lambda G'(a)) \sim \lambda^2 G(a) \,.
\label{eq:6.23dieett}
\end{equation}
  
The following lemma (see \cite[Corollary~26]{DieKre08}) deals with the \emph{change of shift} for $N$-functions.

\begin{lemma}[change of shift]\label{lem:changeshift}
Let $G$ be an $N$-function with $\Delta_2(G),\Delta_2(G^*)<\infty$. Then for any $\eta>0$ there exists $c_\eta>0$, depending only on $\eta$ and $\Delta_2(G)$, such that for all ${\bf a}, {\bf b}\in\R^d$ and $t\geq0$
\begin{equation}
G_{|{\bf a}|}(t) \leq c_\eta G_{|{\bf b}|}(t) + \eta G_{|{\bf a}|}(|{\bf a}-{\bf b}|)\,.
\label{(5.4diekreu)}
\end{equation}
\end{lemma}


Let ${\bf{P}}_0,{\bf{P}}_1\in \mathbb{R}^{N\times n}$, $\theta\in[0,1]$ and define ${\bf{P}}_{\theta}:=(1-\theta){\bf{P}}_0+\theta{\bf{P}}_1$. Then the following result holds (see \cite[Lemma~20]{DIEETT08}).
\begin{lemma}
\label{technisch-mu}
Let $G$ be a $N$-function with $\Delta_2(G, G^*)<\infty.$ Then uniformly for all ${\bf{P}}_0,{\bf{P}}_1\in \mathbb{R}^{N\times n}$ with $|{\bf{P}}_0|+|{\bf{P}}_1|>0$ holds
\begin{equation*}
 \int_0^1 \frac{G'(|{\bf{P}}_{\theta}|)}{|{\bf{P}}_{\theta}|}\, \mathrm{d}\theta \sim \frac{G'(|{\bf{P}}_0|+|{\bf{P}}_1|)}{|{\bf{P}}_0|+|{\bf{P}}_1|}
\end{equation*}
where the constants only depend on $\Delta_2(G, G^*).$
\end{lemma}
In view of the previous considerations, the same proposition holds true for the shifted functions, uniformly in $a\ge 0$.

By $L^\phi$ and $W^{1,\phi}$ we denote the classical Orlicz and Orlicz-Sobolev spaces, i.\,e.\ $f \in L^\phi$ iff $\displaystyle\int
\phi(|{f}|)\,dx < \infty$ and $f \in W^{1,\phi}$ iff $f, D f \in L^\phi$. The space $W^{1,\phi}_0(\Omega)$ will denote the closure of $C^\infty_0(\Omega)$ in $W^{1,\phi}(\Omega)$. The following version of Sobolev-Poincar\'e inequality can be found in \cite[Lemma~7]{DIEETT08}.
\begin{theorem}[Sobolev-Poincar\'e inequality]\label{thm:sob-poincare}
Let $G$ be an $N$-function with $\Delta_2(G,G^*)<+\infty$. Then there exist numbers $\theta=\theta(n,\Delta_2(G,G^*))\in(0,1)$ and $K=K(n,N,\Delta_2(G,G^*))>0$ such that the following holds. If $B\subset \R^n$ is any ball with radius $R$ and ${\bf w}\in W^{1,G}(B,\R^N)$, then
\begin{equation}
\dashint_B G\left(\frac{|{\bf w}-({\bf w})_B|}{R}\right)\,\mathrm{d}x\leq K \left(\dashint_B G^\theta\left({|D{\bf w}|}\right)\,\mathrm{d}x\right)^\frac{1}{\theta}\,,
\label{eq:sob-poincare}
\end{equation}
where $\displaystyle({\bf w})_B:=\dashint_B {\bf w}(x)\,\mathrm{d}x$. Moreover, if ${\bf w}\in W^{1,G}_0(B,\R^N)$, then
\begin{equation*}
\dashint_B G\left(\frac{|{\bf w}|}{R}\right)\,\mathrm{d}x\leq K \left(\dashint_B G^\theta\left({|D{\bf w}|}\right)\,\mathrm{d}x\right)^\frac{1}{\theta}\,,
\label{eq:sob-poincare2}
\end{equation*}
where $K$ and $\theta$ have the same dependencies as before.
\end{theorem}

We conclude this section with the following useful lemma about an almost concave condition.

\begin{lemma}
Let $\Psi:[0,\infty)\to[0,\infty)$ be non-decreasing and such that $t\to \frac{\Psi(t)}{t}$ be non-increasing. Then there exists a concave function $\widetilde{\Psi}: [0,\infty)\to[0,\infty)$ such that
\begin{equation*}
\frac{1}{2} \widetilde{\Psi}(t) \leq \Psi(t) \leq \widetilde{\Psi}(t) \quad \mbox{ for all $t\geq0$. }
\end{equation*}
\label{lem:lemma2.2ok}
\end{lemma}
\begin{proof}
See \cite[Lemma~2.2]{OKJFA18}.   
\end{proof}



\subsection{A lemma of Bogowski\v{\i}} \label{sec:bogow}

The following lemma is a key tool to deal with the constraint of divergence free vector fields, as we have to construct
testing functions in divergence free form. 

\begin{lemma}
\label{lem:bogowski}
Let $B_r(x_0)$ be a ball in $\R^n$ and $f\in L^\gamma(B_r(x_0))$ with $(f)_{x_0,r}=0$, where $1<\gamma_1\leq \gamma \leq \gamma_2<+\infty$. Then there exists ${\bf w}\in W^{1,\gamma}_0(B_r(x_0),\R^n)$, weak solution to ${\rm div}\,{\bf w}=f$ in $B_r(x_0)$, such that
\begin{equation}
\int_{B_r(x_0)} |D{\bf w}|^t\,\mathrm{d}x \leq c(n,\gamma_1,\gamma_2) \int_{B_r(x_0)} |f|^t\,\mathrm{d}x
\label{eq:3.17bog}
\end{equation}
for every $t\in[\gamma_1,\gamma]$.
\end{lemma}

\begin{proof}
See \cite{Bog} and \cite[Chapter~3, Section~3]{galdibook}.
\end{proof}

{\begin{corollary} \label{lem:Hbogowski}
Let $B_r(x_0)$ be a ball in $\R^n$, and $G$ be an $N$-function such that $\Delta_2(G)<+\infty$ and $\Delta_2(G^*)<+\infty$. Let $G(|f|)\in L^1(B_r(x_0))$ with $(f)_{x_0,r}=0$. Then there exists ${\bf w}\in W^{1,G}_0(B_r(x_0),\R^n)$, solution to ${\rm div}\,{\bf w}=f$ a.e. in $B_r(x_0)$, such that
\begin{equation}
\int_{B_r(x_0)} G(|D{\bf w}|)\,\mathrm{d}x \leq c_G \int_{B_r(x_0)} G(|f|)\,\mathrm{d}x\,.
\label{eq:3.17bogtris}
\end{equation}
\end{corollary}
\begin{proof}
See, e.g.,  \cite[Corollary~4.2]{KrepRuz}.
\end{proof}}

%
%

\subsection{Affine functions} \label{sec:affinefunct}

We define the space of \emph{traceless affine functions}

\begin{equation*}
\mathcal{T}(\R^n):=\{\bm\ell:\R^n\to\R^n:\,\,\, \mbox{ $\bm\ell$ is affine and } {\rm tr}(D \bm\ell)=0\}\,.
\end{equation*}

In particular, each $\bm\ell\in \mathcal{T}$ satisfies ${\rm div}\,\bm\ell=0$. \\
\noindent
The set of \emph{rigid motions} in $\R^n$ is defined as
\begin{equation*}
\mathcal{R}:=\{{\bf c}+{\bf S}x:\,\, {\bf c}\in\R^n\,,{\bf S}\in\R^{n\times n}\,, ^T{\bf S}=-{\bf S}\}\,,
\end{equation*}
the set of affine functions with skew-symmetric gradient. 

{Let $x_0\in\Omega$ and $r>0$. We define the affine function
\begin{equation}
\bm\ell_{x_0,r}(x):= (\bfu)_{x_0,r} + (D\bfu)_{x_0,r}(x-x_0)\,.
\label{eq:affinefunctionl}
\end{equation}
Note that $\bm\ell_{x_0,r}\in \mathcal{T}(\R^n)$ and
\begin{equation}
(\bfu - \bm\ell_{x_0,r})_{x_0,r}={\bf 0}\,, \,\,\, \E\bm\ell_{x_0,r}=(\E\bfu)_{x_0,r}\,,\,\,\, (\W(\bfu-\ell_{x_0,r}))_{x_0,r}={\bf 0}\,.
\end{equation}}

Let $\Omega$ be a bounded open subset of $\R^n$ and let $x_0$ be its centroid; i.e.,
\begin{equation*}
x_0 := \frac{1}{|\Omega|} \int_\Omega x\,\mathrm{d}x\,.
\end{equation*}
For every $\bfu\in L^1(\Omega;\R^n)$ and $x\in\Omega$ we set
\begin{equation}
({\bf\mathcal{P}}_\Omega \bfu)_i(x):= c_i + S_{ij}(x-x_0)_j\,,
\label{eq:acmingrigidmotion}
\end{equation}
where 
\begin{equation}
\begin{split}
c_i & := (\bfu)_i = \dashint_\Omega u_i(x)\,\mathrm{d}x\,, \\
S_{ij} & := \frac{\dashint_\Omega\{ [\bfu-(\bfu)]_i(x-x_0)_j-[\bfu-(\bfu)]_j(x-x_0)_i\}\,\mathrm{d}x}{\dashint_\Omega [|(x-x_0)_i|^2+|(x-x_0)_j|^2]\,\mathrm{d}x}\,.
\end{split}
\end{equation} 
As shown in \cite[Proposition~2.6]{acerbimingione}, ${\bf\mathcal{P}}_\Omega \bfu\in\mathcal{R}$ for all $\bfu$; this implies, in particular, that $\E ({\bf\mathcal{P}}_\Omega \bfu)={\bf 0}$. Moreover, if $\bfu\in L^2(\Omega;\R^n)$ and $\Omega$ is a ball, then
\begin{equation}
\|\bfu-{\bf\mathcal{P}}_\Omega \bfu\|_{L^2(\Omega;\R^n)}=\min\{\|\bfu - {\bf r}\|_{L^2(\Omega;\R^n)}:\,\, {\bf r}\in\mathcal{R}\}\,.
\label{eq:Pminimizes}
\end{equation}
{Finally, we recall the following Korn-type inequality \cite[eq. (2.21)]{acerbimingione}:
\begin{equation}
\frac{1}{t}\left(\dashint_{B_r} \left|\frac{\bfu-{\mathcal{P}_{B_r} \bfu}}{r}\right|^t\,\mathrm{d}x\right)^\frac{1}{t} \leq c \left(\dashint_{B_r} \left|\E\bfu\right|^p\,\mathrm{d}x\right)^\frac{1}{p}
\label{eq:2.21acming}
\end{equation}
for every $t\in[1,p^*]$.
}

\subsection{Sobolev-Korn inequality} \label{sec:sobkorn}

A standard tool in order to obtain local bounds of $D\bfu$ in terms of $\E\bfu$ on the scale of $L^p$ spaces is the following Sobolev-Korn inequality (see, e.g., \cite{mos}).

\begin{lemma}
Let $1<p\leq r \leq q$ and $\bfu\in L^p(B_\varrho(x_0),\R^n)$ be such that $\E\bfu\in L^r(B_\varrho(x_0),\R^{n\times n}_{\rm sym})$. Then $D\bfu\in L^r(B_\varrho(x_0),\R^{n^2})$ and for some constant $c=c(n,p,q)$ it holds that
\begin{equation}
\dashint_{B_\varrho(x_0)}|D\bfu|^r\,\mathrm{d}x \leq c \dashint_{B_\varrho(x_0)}|\E\bfu|^r\,\mathrm{d}x + c \left(\dashint_{B_\varrho(x_0)}\left|\frac{\bfu - (\bfu)_{x_0,\varrho}}{\varrho}\right|\,\mathrm{d}x\right)^r\,.
\label{eq:2.9bogel}
\end{equation}
If, in addition, $\bfu={\bf 0}$ on $\partial B_\varrho(x_0)$, then
\begin{equation}
\dashint_{B_\varrho(x_0)}|D\bfu|^r\,\mathrm{d}x \leq c \dashint_{B_\varrho(x_0)}|\E\bfu|^r\,\mathrm{d}x \,.
\label{eq:2.10bogel}
\end{equation}
\end{lemma}

We also need the following version of the Korn's  inequality in Orlicz spaces (see \cite[Lemma~3.3]{DKS}).

\begin{lemma}\label{lem:kornorlicz}
Let $B\subset\R^n$ be a ball. Let $G$ be an $N$-function such that both $G$ and $G^*$ satisfy the $\Delta_2$-condition. Then for all $\bfu\in W^{1,G}_0(B,\R^n)$ (with $\left(\W\bfu\right)_B={\bf 0}$) the inequality
\begin{equation*}
\int_B G(|D\bfu|)\,\mathrm{d}x \leq c_{\rm Korn} \int_B G(|\E\bfu|)\,\mathrm{d}x
\end{equation*}
holds, and $c_{\rm Korn}=c_{\rm Korn}(\Delta_2(G,G^*))$.
\end{lemma}

{The following lemma is useful to derive reverse H\"older estimates. It is a variant of the results by Gehring~\cite{Gehring} and Giaquinta-Modica~\cite[Theorem~6.6]{giustibook}.
\begin{lemma}\label{lem:gehring}
Let $B_0\subset\R^n$ be a ball, $f\in L^1(B_0)$, and $g\in L^{\sigma_0}(B_0)$ for some $\sigma_0>1$. Assume that for some $\theta\in(0,1)$, $c_1>0$ and all balls $B$ with $2B\subset B_0$
\begin{equation*}
\dashint_B |f|\,\mathrm{d}x\leq c_1 \left(\dashint_{2B}|f|^\theta\,\mathrm{d}x\right)^{1/\theta} + \dashint_{2B}|g|\,\mathrm{d}x\,.
\end{equation*}
Then there exist $\sigma_1>1$ and $c_2>1$ such that $g\in L^{\sigma_1}_{\rm loc}(B)$ and for all $\sigma_2\in[1,\sigma_1]$
\begin{equation*}
\left(\dashint_{B}|f|^{\sigma_2}\,\mathrm{d}x\right)^{1/{\sigma_2}}\leq c_2 \dashint_{2B}|f|\,\mathrm{d}x + c_2 \left(\dashint_{2B}|g|^{\sigma_2}\,\mathrm{d}x\right)^{1/{\sigma_2}}\,.
\end{equation*}
\end{lemma}}

\subsection{$\mathcal{A}$-Stokes functions and $\mathcal{A}$-Stokes approximation} \label{sec:Astokes}

{Let $\mathcal{A}$ be a bilinear form on $\R_{\rm sym}^{n\times n}$. We say that $\mathcal{A}$ is \emph{strongly elliptic in the sense of Legendre-Hadamard} if for all $\bm\eta,\bm\zeta\in\R_{\rm sym}^{n\times n}$ it holds that 
\begin{equation}
 \kappa_A\abs{\bm\eta}\abs{\bm\zeta}\leq \mathcal{A}(\bm\eta,\bm\zeta)\leq L_{\mathcal{A}} \abs{\bm\eta} \abs{\bm\zeta}
\label{eq:LegHam}
\end{equation}
for some $L_{\mathcal{A}}\geq \kappaA>0$. The biggest possible constant~$\kappaA$ is called the ellipticity constant of~$\mathcal{A}$. By $\abs{\mathcal{A}}$ we denote the Euclidean norm of~$\mathcal{A}$.  
We say that a Sobolev function $\bfw$ on a ball~$B_\varrho(x_0)$ is
\emph{$\mathcal{A}$-Stokes} on $B_\varrho(x_0)$ if it satisfies $- {\rm div}\,(\mathcal{A}\E \bfw)=0$ in the sense of distributions; i.e.,
\begin{equation*}
\int_{B_\varrho(x_0)}\mathcal{A}(\E{\bf w},\E\bm\psi)\,\mathrm{d}x=0\,,\quad \mbox{ for all }\bm\psi\in C^\infty_{0,{\rm div}}(B_\varrho(x_0),\R^n)\,.
\end{equation*}
Now, we address the problem of the \emph{$\mathcal{A}$-Stokes approximation}: given a Sobolev function $\bfv$ on a ball~$B$, we want to find an $\mathcal{A}$-Stokes function~${\bf h}$ which is ``close'' the function $\bfv$. 
It will be the $\mathcal{A}$-Stokes function with the same boundary values as $\bfv$; i.e., a Sobolev function~${\bf h}$
which satisfies
\begin{equation}
  \label{eq:calA1}
  \begin{cases}
    -{\rm div} (\mathcal{A} \E {\bf h})= 0 &\qquad\text{on $B$}
    \\
    {\bf h}= \bfv & \qquad\text{on $\partial B$}
  \end{cases}
\end{equation}
in the sense of distributions.
Setting ${\bf w} := {\bf h} - \bfv$, then~\eqref{eq:calA1} is equivalent to finding
a Sobolev function~${\bf w}$ which satisfies
\begin{equation}
  \label{eq:calA2}
  \begin{cases}
    -{\rm div}(\mathcal{A} \E {\bf w}) = -{\rm div}(\mathcal{A}
    \E \bfv) &\qquad\text{on $B$}
    \\
    {\bf w}= \bfzero &\qquad\text{on $\partial B$}
  \end{cases}
\end{equation}
in the sense of distributions.}
{We formulate the $\mathcal{A}$-Stokes approximation result for Stokes systems with general growth. The proof  for Stokes systems with $p$-growth can be found in \cite[Theorem~4.2]{BrDieFuchs12}, whence the case of general growth can be inferred by using a duality argument as in \cite{DIELENSTROVER12}, based on the following Lemma. }

\begin{lemma}
  \label{lem:varineqphi}
  Let $\phi$
  be an N-function with $\Delta_2(\phi,\phi^*) < \infty$. Let $B \subset \Rn$ be a ball, and let $\mathcal{A}$ be strongly
  elliptic in the sense of Legendre-Hadarmard. Then 
the following variational estimates
  hold for all $\bfv \in W^{1, \phi}_{0, \dive}(B)$
  \begin{align}
    \label{eq:varest}
    \int_B \phi(\abs{\E\bfv})\,\mathrm{d}x \lesssim\sup_{\bfxi \in
      (C^\infty_{0. \dive}(B))^N} \bigg[ \int_B \mathcal{A}(\E \bfv,
    \E\bfxi)\,\mathrm{d}x - \int_B \phi^*(\abs{D \bfxi})\,\mathrm{d}x
    \bigg]\,.
  \end{align}
The implicit constants in \eqref{eq:varest} 
only depend on $n$, $\kappaA$,
  $\norm{\mathcal{A}}$ and $\Delta_2(\phi,\phi^*)$.
\end{lemma}
\begin{proof}
{This result can be deduced from \cite[Lemma~4.1]{BrDieFuchs12} along the lines of the proof of \cite[Lemma~20]{DIELENSTROVER12}. We then omit the details. }
\end{proof}

The $\mathcal{A}$-Stokes approximation can be stated as follows.
\begin{theorem}\label{AStokes}
  
  Let $B \subset \Rn$ be a ball with radius~$r_B$ and let
  $\widetilde{B} \subset \Rn$ denote either~$B$ or $2B$.  Let $\phi$
  be an N-function with $\Delta_2(\phi,\phi^*) < \infty$ and let
  $s>1$. Then for every $\kappa>0$, there exists $\delta>0$ only
  depending on $n$, $\kappa_A$, $\abs{\mathcal{A}}$,
  $\Delta_2(\phi,\phi^*)$ and $s$ such that the following holds.
  If $\bfv \in W^{1,\phi}_{\dive}(\widetilde{B})$ is {\em almost
    $\mathcal{A}$-Stokes } on~$B$ in the sense that
  \begin{align}
    \label{eq:Aappr_ah}
    \biggabs{\dashint_B \mathcal{A}(\E\bfv, \E\bfxi)\,\mathrm{d}x}
    \leq \delta \dashint_{\widetilde{B}} \abs{\E\bfv}\,\mathrm{d}x
    \norm{D \bfxi}_\infty
  \end{align}
  for all $\bfxi \in C^\infty_{0, \dive}(B)$. Then the unique solution $\bfw
  \in (W^{1, \phi_s}_{0,\dive}(B))^n$ of~\eqref{eq:calA2} satisfies

  \begin{align}
    \label{eq:Aappr_est}
    \dashint_B \phi\bigg(\frac{\abs{\bfw}}{r_B}\bigg)\,\mathrm{d}x +
    \dashint_B \phi(\abs{D\bfw})\,\mathrm{d}x \leq \kappa
    \bigg(\dashint_{\widetilde{B}} \big(\phi(\abs{D
      \bfv})\big)^s \,\mathrm{d}x\bigg)^{\frac 1s}.
  \end{align}
It holds $\kappa=\kappa(\phi, s, \delta)$ and $\lim_{\delta\to 0} \kappa(\phi,s,\delta)=0$. The function $\bfh=\bfv-\bfw$  is called the $\mathcal{A}$-Stokes approximation of $\bfv.$
\end{theorem}


{
\begin{remark}\label{rem:thmmodified}
We will exploit the previous approximation result in a slightly modified version. Indeed, 
under the additional assumption $(\mathcal{W}{\bf v})_{\tilde{B}}={\bf 0}$ and
\begin{equation}
\dashint_{\tilde{B}} G(|\E{\bf v}|)\,\mathrm{d}x \leq \left(\dashint_{\tilde{B}} [G(|\E{\bf v}|)]^s\,\mathrm{d}x\right)^{\frac{1}{s}}\leq G(\epsilon)
\label{eq:alternativebound}
\end{equation}
for some exponent $s>1$ and for a constant $\epsilon>0$, and \eqref{eq:Aappr_ah} replaced by
  \begin{equation*}
        \biggabs{\dashint_B \mathcal{A}(\E\bfv, \E\bfxi)\,\mathrm{d}x}
    \leq \delta \epsilon
    \norm{D \bm\xi}_{L^\infty(B)}\,,
  \end{equation*}
by using Korn's inequality (Lemma~\ref{lem:kornorlicz}) and following \cite[Lemma~2.7]{CeladaOk}, it can be seen 
that the unique solution ${\bf w}
  \in (W^{1, \phi_s}_{0,\dive}(B))^n$ of~\eqref{eq:calA2} satisfies
\begin{equation}
\dashint_B G\bigg(\frac{\abs{{\bf w}}}{r_B}\bigg)\,\mathrm{d}x +
    \dashint_B G(\abs{D{\bf w}})\,\mathrm{d}x \leq \kappa G(\epsilon)\,.
\label{eq:lemma2.7bis}
\end{equation}
\label{rem:remarkastokes}
\end{remark}
}

%

%

\subsection{Self-improving properties of weak solutions}\label{sec:selfimprove}
{We prove the self improving properties for double phase Stokes systems \eqref{eq:1.1ok} with lower order terms. {These are some higher integrability results which will play a crucial role in the sequel. They can be compared with analogous results obtained for Stokes systems with $p(\cdot)$-growth, see \cite[Theorem~4.2]{acerbimingione} and \cite[Lemma~2.10]{BDHS 12}. In this direction, our main difficulty is in handling the lower order terms without any ``closeness'' condition between the extreme exponents $p$ and $q$, which, on the contrary, can be ensured in the variable exponent setting by requiring some continuity assumption on $p(\cdot)$. For instance, in the main estimate \eqref{eq:3.6ok}, we are forced to control the average of the term $r^\frac{p}{2}|D\bfu|^p$ instead of the more natural $r^{p}|D\bfu|^p$ appearing in \cite[Lemma~2.10]{BDHS 12}. Nevertheless, these issues can be overcome when dealing with the terms involving $H(x,t)$, just playing with the different behavior of $H$ according to the magnitude of $\mu$ inside a small ball.} 
\begin{lemma}\label{lem:higint}
Let $H:\Omega \times [0,\infty)\to[0,\infty)$ be defined as in \eqref{eq:H}, and satisfying \eqref{eq:pq}. Let ${\bf a}:\Omega\times\R_{\rm sym}^{n\times n}\to \R_{\rm sym}^{n\times n}$ be such that
\begin{equation}
|{\bf a}(x,\bm\xi)|\leq L H'(x, 1+|\bm\xi|) \quad \mbox{ and }\quad {\bf a}(x,\bm\xi):\bm\xi\geq \nu H(x,|\bm\xi|) - \nu_0 H(x,1)
\label{eq:3.5ok}
\end{equation}
for all $x\in\Omega$ and $\bm\xi\in\R_{\rm sym}^{n\times n}$, and for some $0<\nu\leq L<\infty$, $\nu_0\in[0,1]$.
Let $\bfu\in W^{1,1}(\Omega;\R^n)$ with $H(\cdot,|\E\bfu|)\in L^1(\Omega)$ be a weak solution to \eqref{eq:1.1ok}. Then there exists $s_0>1$, depending on $n,p,q, [\mu]_{C^\alpha},\nu,L,\nu_0, \|\E\bfu\|_{L^p(\Omega)}$, such that $H(\cdot,|\E\bfu|)\in L^{s_0}_{\rm loc}(\Omega)$ and,  for every $s\in(1,s_0]$ and every $B_{2r}(x_0)\subset\subset\Omega$, $r<1$,
\begin{equation}
\left(\dashint_{B_{r}(x_0)}[H(x,1+|\E\bfu|)]^{s}\,\mathrm{d}x\right)^{\frac{1}{s}} \leq c \dashint_{B_{2r}(x_0)}H(x,1+|\E\bfu|)\,\mathrm{d}x + c\dashint_{B_{2r}(x_0)}(r^\frac{p}{2}|D\bfu|^p+|\bfu|^{p^*}+1) \, \mathrm{d}x
\label{eq:3.6ok}
\end{equation}
for some $c=c(n,p,q, [\mu]_{C^\alpha},\nu,L,\nu_0, \|\E\bfu\|_{L^p(\Omega)})>0$.

{Moreover, for each $t\in[0,1]$ we have
\begin{equation}
\left(\dashint_{B_{r}(x_0)}[H(x,1+|\E\bfu|)]^{s}\,\mathrm{d}x\right)^{\frac{1}{s}} \leq c_t \left(\dashint_{B_{2r}(x_0)}[H(x,1+|\E\bfu|)]^t\,\mathrm{d}x\right)^\frac{1}{t} + c\dashint_{B_{2r}(x_0)}(r^\frac{p}{2}|D\bfu|^p+|\bfu|^{p^*}+1) \, \mathrm{d}x
\label{eq:3.6okbis}
\end{equation}
for some constant $c_t = c_t(n,p,q, [\mu]_{C^\alpha},\nu,L,\nu_0, \|\E\bfu\|_{L^p(\Omega)},t)>0$.}
\end{lemma}
\begin{proof}
{Our argument partly follows that of \cite[Theorem~3.4]{OKJFA18}, devised for homogeneous double-phase systems depending on the gradient $D\bfu$, while the estimates for the right hand side's terms in \eqref{eq:system2} are essentially taken from Acerbi and Mingione \cite[Theorem~4.2]{acerbimingione}. }

Let $\eta\in C_0^\infty(B_{2r}(x_0))$ be a cut-off function such that $\eta\equiv1$ on $B_{r}(x_0)$, $0\leq\eta\leq1$ and $|D\eta|\leq \frac{c}{2r}$, and let $ {\bf \mathcal{P}}:={\bf \mathcal{P}}_{B_{2r}(x_0)}$ be the rigid displacement introduced in \eqref{eq:acmingrigidmotion}. We consider
\begin{equation}
\bm\psi:=\eta^q({\bf u} - {\bf \mathcal{P}}) + {\bf w}\,,
\end{equation}
where the function $\bf w$ is defined through Lemma~\ref{lem:bogowski} as a solution to
\begin{equation}
{\rm div}\, {\bf w} =  - {\rm div}\,(\eta^q({\bf u} - {\bf \mathcal{P}}))= - q\eta^{q-1}D\eta\cdot ({\bf u} - {\bf \mathcal{P}}) \,.
\label{eq:4.7acming}
\end{equation}
Such a $\bf w$ exists since ${\rm div}\,\bfu=0$ and 
\begin{equation*}
\int_{B_{2r}(x_0)}({\bf u} - {\bf \mathcal{P}}) D(\eta^q)\,\mathrm{d}x =0\,, 
\end{equation*}
and by \eqref{eq:pq} and the summability properties of $\bfu$ it holds that ${\bf w}\in W^{1,q}_0(B_{2r}(x_0);\R^n)$. We recall also that, by \eqref{eq:3.17bog}, we have the estimate
\begin{equation}
\dashint_{B_{2r}(x_0)} |D{\bf w}|^t\,\mathrm{d}x \leq c \dashint_{B_{2r}(x_0)}\left|\frac{\bfu - {\bf \mathcal{P}}}{2r}\right|^t\,\mathrm{d}x
\label{eq:4.8acming}
\end{equation}
for every exponent $t$ for which the right hand side is finite.
Taking $\bm\psi$ as a test function in \eqref{eq:1.1ok} we get
\begin{equation}
\begin{split}
J_1:= \dashint_{B_{2r}} \eta^q {\bf a}(x,\E\bfu):\E{\bf u}\,\mathrm{d}x & = -q\dashint_{B_{2r}}\eta^{q-1} {\bf a}(x,\E\bfu):(D \eta\odot({\bf u} - {\bf \mathcal{P}}))\,\mathrm{d}x \\
& \,\,\,\,\,\, - \dashint_{B_{2r}} {\bf a}(x,\E\bfu):\E{\bf w}\,\mathrm{d}x + \dashint_{B_{2r}} \eta^q \bfu[D \bfu]({\bf u} - {\bf \mathcal{P}})\,\mathrm{d}x \\
& \,\,\,\,\,\, + \dashint_{B_{2r}} \bfu[D \bfu]{\bf w}\,\mathrm{d}x + \dashint_{B_{2r}} \eta^q f({\bf u} - {\bf \mathcal{P}})\,\mathrm{d}x + \dashint_{B_{2r}} f{\bf w}\,\mathrm{d}x \\
& =: J_2+J_3+J_4+J_5+J_6+J_7\,.
\end{split}
\end{equation}
By the second condition in \eqref{eq:3.5ok} we have
\begin{equation}
J_1 \geq \nu \dashint_{B_{2r}} \eta^q H(x,|\E \bfu|)\,\mathrm{d}x - \nu_0 \dashint_{B_{2r}} \eta^qH(x,1)\,\mathrm{d}x \,,
\end{equation}
while from the first one we get
\begin{equation}
J_2 \leq c \dashint_{B_{2r}} \eta^q H'(x, 1 + |\E\bfu|) \frac{|\bfu - {\bf \mathcal{P}}|}{2r}\,\mathrm{d}x\,.
\end{equation} 
Now, by using Young's inequality, for any $\kappa\in(0,1)$ we obtain
\begin{equation*}
\begin{split}
\dashint_{B_{2r}}\eta^{q-1} |\E\bfu|^{p-1} \frac{|\bfu - {\bf \mathcal{P}}|}{2r} \,\mathrm{d}x & \leq \kappa \dashint_{B_{2r}}\eta^{q} |\E\bfu|^{p}\,\mathrm{d}x + c_\kappa \dashint_{B_{2r}}\left(\frac{|\bfu - {\bf \mathcal{P}}|}{2r}\right)^p\,\mathrm{d}x\,, \\
\dashint_{B_{2r}}\eta^{q-1} \mu(x)|\E\bfu|^{q-1} \frac{|\bfu - {\bf \mathcal{P}}|}{2r}\,\mathrm{d}x & \leq \kappa \dashint_{B_{2r}}\eta^{q} \mu(x)|\E\bfu|^{q}\,\mathrm{d}x + c_\kappa \dashint_{B_{2r}}\mu(x)\left(\frac{|\bfu - {\bf \mathcal{P}}|}{2r}\right)^q\,\mathrm{d}x\,,
\end{split}
\end{equation*}
whence
\begin{equation*}
\dashint_{B_{2r}} \eta^q H'(x, 1 + |\E\bfu|) \frac{|\bfu - {\bf \mathcal{P}}|}{2r}\,\mathrm{d}x \leq \kappa \dashint_{B_{2r}} \eta^q H(x, 1+ |\E \bfu|)\,\mathrm{d}x + c_\kappa\dashint_{B_{2r}}H\left(x, \frac{|\bfu - {\bf \mathcal{P}}|}{2r}\right) \,\mathrm{d}x \,.
\end{equation*}
As for $J_3$, using again the first condition in \eqref{eq:3.5ok} and Young's inequality, and Corollary~\ref{lem:Hbogowski}, we have
\begin{equation*}
J_3 \leq c \dashint_{B_{2r}} H'(x, 1 + |\E\bfu|) |\E {\bf w}|\,\mathrm{d}x \leq \kappa \dashint_{B_{2r}} H(x, 1 + |\E \bfu|)\,\mathrm{d}x + c_\kappa\dashint_{B_{2r}}H\left(x, \frac{|\bfu - {\bf \mathcal{P}}|}{2r}\right) \,\mathrm{d}x \,.
\end{equation*}
We can write
\begin{equation*}
\begin{split}
J_4+J_6 & \leq \dashint_{B_{2r}}|\bfu||D\bfu||{\bf u} - {\bf \mathcal{P}}|\,\mathrm{d}x + \dashint_{B_{2r}} |f||{\bf u} - {\bf \mathcal{P}}|\,\mathrm{d}x \\
J_5+J_7 & \leq \dashint_{B_{2r}}|\bfu||D\bfu||{\bf w}|\,\mathrm{d}x + \dashint_{B_{2r}} |f||{\bf w}|\,\mathrm{d}x \,.
\end{split}
\end{equation*}
Let $0<\beta<\frac{1}{n+2}$. We set
\begin{equation*}
\gamma:= 1+ \frac{\beta}{2}\left(\frac{n+2}{n}\right)\,, \quad \sigma:= \left[\frac{1}{2}\left(\frac{p}{\gamma}\right)^*\right]'\,,
\end{equation*}
and note that 
\begin{equation}
1\leq \frac{p}{\gamma}<n\,, \quad 1<\sigma<\sigma \gamma \leq p\,, \quad \left(\frac{p}{\gamma}\right)^* \leq \frac{p^*}{\gamma}\,, 
\label{eq:4.10acming}
\end{equation}
since
\begin{equation*}
2\beta + \frac{3n}{n+2} \leq p < n < n\gamma\,.
\end{equation*}
{By using H\"older's inequality, the Poincar\'e inequality for $\bfu-{\bf \mathcal{P}}$ and \eqref{eq:4.10acming} 
we have
\begin{equation*}
\begin{split}
&\dashint_{B_{2r}} |\bfu| |D\bfu||\eta^q({\bf u} - {\bf \mathcal{P}})|\,\mathrm{d}x \\
 & \,\,\,\,\,\,\leq 2r\left(\dashint_{B_{2r}} |\bfu|^{(\frac{p}{\gamma})^*}\,\mathrm{d}x\right)^{\frac{1}{(\frac{p}{\gamma})^*}}\left(\dashint_{B_{2r}} |D\bfu|^{\sigma}\,\mathrm{d}x\right)^\frac{1}{\sigma}\left(\dashint_{B_{2r}} \left|\frac{{\bf u} - {\bf \mathcal{P}}}{2r}\right|^{(\frac{p}{\gamma})^*}\,\mathrm{d}x\right)^{\frac{1}{(\frac{p}{\gamma})^*}} \\
& \,\,\,\,\,\,\leq cr \left(\dashint_{B_{2r}} |\bfu|^{(\frac{p}{\gamma})^*}\,\mathrm{d}x\right)^{\frac{1}{(\frac{p}{\gamma})^*}}\left(\dashint_{B_{2r}} |D\bfu|^{\sigma}\,\mathrm{d}x\right)^\frac{1}{\sigma} \left(\dashint_{B_{2r}} |D\bfu|^{\frac{p}{\gamma}}\,\mathrm{d}x\right)^\frac{\gamma}{p}  \\
& \,\,\,\,\,\,\leq cr \left(\dashint_{B_{2r}} |\bfu|^{\frac{p^*}{\gamma}}\,\mathrm{d}x\right)^{\frac{\gamma}{p^*}}\left(\dashint_{B_{2r}} |D\bfu|^{\frac{p}{\gamma}}\,\mathrm{d}x\right)^\frac{2\gamma}{p} 
\end{split}
\end{equation*}
whence, with Young's inequality, 
\begin{equation*}
\begin{split}
\dashint_{B_{2r}} |\bfu| |D\bfu||\eta^q({\bf u} - {\bf \mathcal{P}})|\,\mathrm{d}x & \leq c \dashint_{B_{2r}} |\bfu|^{\frac{p^*}{\gamma}}\,\mathrm{d}x + c r^{\frac{np}{np-n\gamma+p\gamma}}\left(\dashint_{B_{2r}} |D\bfu|^{\frac{p}{\gamma}}\,\mathrm{d}x\right)^\frac{2n\gamma}{np-n\gamma+p\gamma} \\
& \leq c\left(\dashint_{B_{2r}} |\bfu|^{\frac{p^*}{\gamma}}\,\mathrm{d}x + r^{\frac{p}{2\gamma}}\dashint_{B_{2r}} |D\bfu|^{\frac{p}{\gamma}}\,\mathrm{d}x + 1 \right)\,.
\end{split}
\end{equation*} }
Again using H\"older's and Young's inequalities, \eqref{eq:2.21acming} and the fact that \eqref{eq:4.10acming} implies $1^*\leq (\frac{p}{\gamma})^*$, we obtain
\begin{equation*}
\begin{split}
\dashint_{B_{2r}} |f||{\bf u} - {\bf \mathcal{P}}|\,\mathrm{d}x & \leq c \left(\dashint_{B_{2r}} |f|^n\,\mathrm{d}x\right)^\frac{1}{n} \left(\dashint_{B_{2r}} |{\bf u} - {\bf \mathcal{P}}|^{\frac{n}{n-1}}\,\mathrm{d}x\right)^\frac{n-1}{n} \\
& \leq c \|f\|_{L^{n}(\Omega)} \left(\dashint_{B_{2r}} \left|\frac{{\bf u} - {\bf \mathcal{P}}}{2r}\right|^{p^*}\,\mathrm{d}x\right)^\frac{1}{p^*}\\
& \leq c \|f\|_{L^{n}(\Omega)} \left(\dashint_{B_{2r}} |\E\bfu|^{p}\,\mathrm{d}x\right)^\frac{1}{p} \\
& \leq \kappa \dashint_{B_{2r}} |\E\bfu|^{p}\,\mathrm{d}x + c_\kappa \\
& \leq \kappa \dashint_{B_{2r}} H(x,|\E\bfu|)\,\mathrm{d}x + c_\kappa
\end{split}
\end{equation*}
for every $\kappa\in(0,1)$. Using H\"older's inequality, recalling that ${\bf w}\in W^{1,q}_0(B_{2r};\R^n)$, using Poincar\'e inequality, \eqref{eq:4.8acming} and then Young's inequality we obtain the estimate
{\begin{equation*}
\begin{split}
&\dashint_{B_{2r}} |\bfu| |D\bfu||{\bf w}|\,\mathrm{d}x \\
 & \,\,\,\,\,\,\leq \left(\dashint_{B_{2r}} |\bfu|^{(\frac{p}{\gamma})^*}\,\mathrm{d}x\right)^{\frac{1}{(\frac{p}{\gamma})^*}}\left(\dashint_{B_{2r}} |D\bfu|^{\sigma}\,\mathrm{d}x\right)^\frac{1}{\sigma}\left(\dashint_{B_{2r}} |\bfw|^{(\frac{p}{\gamma})^*}\,\mathrm{d}x\right)^{\frac{1}{(\frac{p}{\gamma})^*}} \\
 & \,\,\,\,\,\,\leq cr\left(\dashint_{B_{2r}} |\bfu|^{(\frac{p}{\gamma})^*}\,\mathrm{d}x\right)^{\frac{1}{(\frac{p}{\gamma})^*}}\left(\dashint_{B_{2r}} |D\bfu|^{\sigma}\,\mathrm{d}x\right)^\frac{1}{\sigma}\left(\dashint_{B_{2r}} |D{\bf w}|^{\frac{p}{\gamma}}\,\mathrm{d}x\right)^{\frac{\gamma}{p}} \\
 & \,\,\,\,\,\,\leq cr\left(\dashint_{B_{2r}} |\bfu|^{(\frac{p}{\gamma})^*}\,\mathrm{d}x\right)^{\frac{1}{(\frac{p}{\gamma})^*}}\left(\dashint_{B_{2r}} |D\bfu|^{\sigma}\,\mathrm{d}x\right)^\frac{1}{\sigma}\left(\dashint_{B_{2r}} \left|\frac{{\bf u} - {\bf \mathcal{P}}}{2r}\right|^{(\frac{p}{\gamma})^*}\,\mathrm{d}x\right)^{\frac{1}{(\frac{p}{\gamma})^*}} \\
&  \,\,\,\,\,\,\leq cr \left(\dashint_{B_{2r}} |\bfu|^{\frac{p^*}{\gamma}}\,\mathrm{d}x\right)^{\frac{\gamma}{p^*}}\left(\dashint_{B_{2r}} |D\bfu|^{\frac{p}{\gamma}}\,\mathrm{d}x\right)^\frac{2\gamma}{p} \\ 
& \leq c\left(\dashint_{B_{2r}} |\bfu|^{\frac{p^*}{\gamma}}\,\mathrm{d}x + r^{\frac{p}{2\gamma}}\dashint_{B_{2r}} |D\bfu|^{\frac{p}{\gamma}}\,\mathrm{d}x + 1 \right)\,.
\end{split}
\end{equation*}}
Arguing as above and with Korn's inequality we get
\begin{equation*}
\begin{split}
\dashint_{B_{2r}} |f||{\bf w}|\,\mathrm{d}x & \leq c \left(\dashint_{B_{2r}} |f|^n\,\mathrm{d}x\right)^\frac{1}{n} \left(\dashint_{B_{2r}} |{\bf w}|^{p^*}\,\mathrm{d}x\right)^\frac{1}{p^*} \\
& \leq c r \|f\|_{L^{n}(\Omega)} \left(\dashint_{B_{2r}}|D{\bf w}|^{p}\,\mathrm{d}x\right)^\frac{1}{p} \\
& \leq c \|f\|_{L^{n}(\Omega)} \left(\dashint_{B_{2r}}\left|\frac{\bfu - {\bf \mathcal{P}}}{2r}\right|^p\,\mathrm{d}x\right)^\frac{1}{p}  \\
& \leq c \|f\|_{L^{n}(\Omega)} \left(\dashint_{B_{2r}}|\E \bfu|^p\,\mathrm{d}x\right)^\frac{1}{p} \\
& \leq \kappa \dashint_{B_{2r}} |\E\bfu|^{p}\,\mathrm{d}x + c_\kappa \\
& \leq \kappa \dashint_{B_{2r}} H(x,|\E\bfu|)\,\mathrm{d}x + c_\kappa \,,
\end{split}
\end{equation*}
where in the last inequality we applied Young's inequality for some $\kappa\in(0,1)$.
Collecting the previous estimates and taking into account \eqref{eq:4.10acming}, we finally have
\begin{equation}
\begin{split}
\dashint_{B_{r}} H(x,1+|\E \bfu|)\,\mathrm{d}x & \leq \kappa  \dashint_{B_{2r}} H(x,1+|\E \bfu|)\,\mathrm{d}x + c_\kappa\dashint_{B_{2r}}H\left(x, \frac{|\bfu - {\bf \mathcal{P}}|}{2r}\right) \,\mathrm{d}x \\
& \,\,\,\,\,\, + c_\kappa \dashint_{B_{2r}(x_0)}(r^{\frac{p}{2\gamma}}|D\bfu|^\frac{p}{\gamma}+|\bfu|^\frac{p^*}{\gamma}+1) \, \mathrm{d}x\,,
\end{split}
\label{eq:4.11acming}
\end{equation}
for every $\kappa\in(0,1)$. Note that setting
\begin{equation*}
g:= r^{\frac{p}{2\gamma}} |D\bfu|^\frac{p}{\gamma}+|\bfu|^\frac{p^*}{\gamma}+1
\end{equation*}
we have $g\in L^\gamma(B_{2r})$, since by Korn's inequality and the Sobolev embedding theorem it holds that $|D\bfu|^{p}+|\bfu|^{p^*}\in L^1(\Omega)$.
Now we claim that there exists $\theta=\theta(n,p,q)\in(0,1)$ such that
\begin{equation}
\dashint_{B_{2r}}H\left(x, \frac{|\bfu - {\bf \mathcal{P}}|}{2r}\right) \,\mathrm{d}x \leq c  \left(\dashint_{B_{2r}} [H(x,|\E \bfu|)]^{\theta}\,\mathrm{d}x \right)^\frac{1}{\theta}\,,
\label{eq:4.6acming}
\end{equation}
where $c=c(q,p,\mu, \|\E \bfu\|_{L^p(\Omega)})>0$.
{We first note that, setting $q_*:=\frac{nq}{n+q}$, we have $q_*<p$ by \eqref{eq:pq}. Then, by Korn's inequality 
it holds that
\begin{equation}
\dashint_{B_{2r}}\left|\frac{\bfu - {\bf \mathcal{P}}}{2r}\right|^q\,\mathrm{d}x \leq c \left(\dashint_{B_{2r}} |\E \bfu|^{q_*}\,\mathrm{d}x \right)^{\frac{q}{q_*}}\,.
\label{eq:4.6acmingbis}
\end{equation}
We distinguish between two cases. If $\sup_{B_{2r}}\mu(\cdot) \leq 4 [\mu]_{C^\alpha}r^\alpha$, using \eqref{eq:4.6acmingbis} and H\"older's inequality we then find
\begin{equation*}
\begin{split}
\dashint_{B_{2r}}\mu(x)\left|\frac{\bfu - {\bf \mathcal{P}}}{2r}\right|^q\,\mathrm{d}x & \leq c 4 [\mu]_{C^\alpha}r^\alpha\left(\dashint_{B_{2r}} |\E \bfu|^{q_*}\,\mathrm{d}x \right)^{\frac{q}{q_*}} \\
   & \leq c 4 [\mu]_{C^\alpha}r^\alpha \left(\dashint_{B_{2r}} |\E \bfu|^{p}\,\mathrm{d}x \right)^{\frac{q-p}{p}} \left(\dashint_{B_{2r}} |\E \bfu|^{q_*}\,\mathrm{d}x \right)^\frac{p}{q_*}  \\
& \leq  c [\mu]_{C^\alpha}\|\E \bfu\|_{L^p(\Omega)}^{q-p}r^{\alpha - \frac{n(q-p)}{p}} \left(\dashint_{B_{2r}} |\E \bfu|^{q_*}\,\mathrm{d}x \right)^\frac{p}{q_*}\,,
\end{split}
\end{equation*}
whence \eqref{eq:4.6acming} follows by choosing $\theta:=\frac{q_*}{p}<1$ and noting that, by Korn's inequality,
\begin{equation*}
\begin{split}
\dashint_{B_{2r}}\left|\frac{\bfu - {\bf \mathcal{P}}}{2r}\right|^p\,\mathrm{d}x & \leq c \left(\dashint_{B_{2r}} |\E \bfu|^{p_*}\,\mathrm{d}x \right)^{\frac{p}{p_*}}\\
& \leq c \left(\dashint_{B_{2r}} |\E \bfu|^{p\theta}\,\mathrm{d}x \right)^{\frac{1}{\theta}}  \,,
\end{split}
\end{equation*}
where in the latter we used $p_*\leq q_*$.} If instead $\inf_{B_{2r}}\mu(\cdot)>3[\mu]_{C^\alpha}r^\alpha$, then we can freeze $\mu$ at $x=x_0$ and apply Theorem~\ref{thm:sob-poincare} and Lemma~\ref{lem:kornorlicz} with $G(t):=H(x_0,t)$, since in this case $\frac{3}{4} H(x,t)\leq H(x_0,t)\leq \frac{4}{3} H(x,t)$ for every $x\in B_{2r}(x_0)$. This concludes the proof of the claim.

Now, combining \eqref{eq:4.11acming} and \eqref{eq:4.6acming} we can write
\begin{equation*}
\begin{split}
\dashint_{B_{r}} H(x,1+|\E \bfu|)\,\mathrm{d}x & \leq \kappa  \dashint_{B_{2r}} H(x,1+|\E \bfu|)\,\mathrm{d}x + c_\kappa \left(\dashint_{B_{2r}} [H(x,1+|\E \bfu|)]^{\theta}\,\mathrm{d}x \right)^\frac{1}{\theta} + c_\kappa \dashint_{B_{2r}(x_0)}g\, \mathrm{d}x\,.
\end{split}
\end{equation*}
This estimate holds for every $\kappa\in(0,1)$ and every ball $B_{2r}\subset\subset\Omega$, and the constants $c_\kappa$ only depend on the data. Thus, by virtue of a variant of Gehring's lemma (Lemma~\ref{lem:gehring}) we get \eqref{eq:3.6ok}. This concludes the proof.
\end{proof}

We define
\begin{equation}
\varepsilon_0:= \varepsilon_0(n, p,q, [\mu]_{C^\alpha},\nu,L,\|\E\bfu\|_{L^p(\Omega)}) := \min\left\{\frac{\alpha}{2}, \frac{n(q-p)(s_0-1)}{2ps_0}\right\}\,,
\label{eq:3.7ok}
\end{equation}
where $s_0$ is that of Lemma~\ref{lem:higint}.
We now state an higher integrability estimate near the $p$-phase.
\begin{lemma}
\label{lem:lemma3.6ok}
For $B_{2r}(x_0)\subset\subset\Omega$ with $r\leq1$, if 
\begin{equation}
\inf_{B_{2r}}\mu(\cdot) \leq [\mu]_{C^\alpha}(2r)^{\alpha-\varepsilon_0}\,,
\label{eq:3.8ok}
\end{equation}
then
\begin{equation}
\left(\dashint_{B_{r}(x_0)}[H(x,1+|\E\bfu|)]^{s_0}\,\mathrm{d}x\right)^{\frac{1}{s_0}} \leq c (1+2^\alpha r^{\varepsilon_0})\dashint_{B_{2r}(x_0)}(1+|\E\bfu|)^p\,\mathrm{d}x + c\dashint_{B_{2r}(x_0)}(r^\frac{p}{2}|D\bfu|^p+|\bfu|^{p^*}+1) \, \mathrm{d}x
\label{eq:3.9ok}
\end{equation}
for some $c=c(n,p,q, [\mu]_{C^\alpha},\nu,L,\nu_0,\|\E\bfu\|_{L^p(\Omega)}, \|1+|\E\bfu|\|_{L^{ps_0}(B_{2r}(x_0))})$.
\noindent
{If, in addition,
\begin{equation}
r^\frac{p}{2} (|D\bfu|^p)_{x_0,2r} + |(\bfu)_{x_0,2r}| \leq M\,,
\label{eq:smallnesshighint}
\end{equation}
then
\begin{equation}
\left(\dashint_{B_{r}(x_0)}[H(x,1+|\E\bfu|)]^{s_0}\,\mathrm{d}x\right)^{\frac{1}{s_0}} \leq c (1+2^\alpha r^{\varepsilon_0})\dashint_{B_{2r}(x_0)}(1+|\E\bfu|)^p\,\mathrm{d}x \,.
\label{eq:3.9okbis}
\end{equation}}
\end{lemma}


\begin{proof}
The argument is similar to that of \cite[Lemma~3.6]{OKNLA18}. Thus, we only sketch the proof. 

We first note that under assumption \eqref{eq:3.8ok}, for every $x\in B_{2r}(x_0)$ we have $\mu(x)\leq 3[\mu]_{C^\alpha}r^{\alpha-\varepsilon_0}$. This fact together with \eqref{eq:3.6okbis}, applied with $t=\frac{t_0}{2}\in(0,1)$ where
\begin{equation}
t_0:= \min\left\{\frac{2p}{q}, \frac{p}{q-p}\right\}\,,
\label{eq:3.11oknla}
\end{equation}
implies
\begin{equation}
\begin{split}
\dashint_{B_r(x_0)} \mu(x)(1+|\E\bfu|)^q \,\mathrm{d}x & \leq c \left(\dashint_{B_{2r}(x_0)}(1+|\E\bfu|)^\frac{pt_0}{2}\right)^\frac{2}{t_0} + c r^{\alpha-\varepsilon_0} \left(\dashint_{B_{2r}(x_0)}(1+|\E\bfu|)^\frac{qt_0}{2}\,\mathrm{d}x\right)^\frac{2}{t_0} \\
& \,\,\,\, + c\dashint_{B_{2r}(x_0)}(r^\frac{p}{2}|D\bfu|^p+|\bfu|^{p^*}+1) \, \mathrm{d}x \,.
\end{split}
\label{eq:3.12oknla}
\end{equation}
In \eqref{eq:3.12oknla}, the second integral is finite since $(qt_0)/2\leq p$ by \eqref{eq:3.11oknla}. Moreover, with H\"older's inequality and the fact that $(q-p)t_0\leq p$ again by \eqref{eq:3.11oknla}, we have
\begin{equation}
 \left(\dashint_{B_{2r}(x_0)}(1+|\E\bfu|)^\frac{qt_0}{2}\right)^\frac{2}{t_0} \leq c r^{-\frac{(q-p)n}{ps_0}} \|1+|\E\bfu|\|_{L^{ps_0}(B_{2r}(x_0))}^{q-p} \dashint_{B_{2r}(x_0)}(1+|\E\bfu|)^p\,\mathrm{d}x \,.
\label{eq:3.13oknla}
\end{equation}
Therefore, plugging \eqref{eq:3.12oknla} into \eqref{eq:3.13oknla} and using $r\leq1$, \eqref{eq:pq} 
and the definition of $\varepsilon_0$, we get
\begin{equation}
\begin{split}
\dashint_{B_r(x_0)} \mu(x)(1+|\E\bfu|)^q \,\mathrm{d}x & \leq c \left(1+ 2^\alpha r^{\alpha-\varepsilon_0-\frac{(q-p)n}{p}-\frac{(q-p)n(s_0-1)}{ps_0}}\right)\dashint_{B_{2r}(x_0)}(1+|\E\bfu|)^p \\
& \,\,\,\, + c\dashint_{B_{2r}(x_0)}(r^\frac{p}{2}|D\bfu|^p+|\bfu|^{p^*}+1) \, \mathrm{d}x \\
& \leq c \left(1+ 2^\alpha r^{\varepsilon_0}\right)\dashint_{B_{2r}(x_0)}(1+|\E\bfu|)^p  + c\dashint_{B_{2r}(x_0)}(r^\frac{p}{2}|D\bfu|^p+|\bfu|^{p^*}+1) \, \mathrm{d}x \,.
\end{split}
\label{eq:3.14oknla}
\end{equation}
This concludes the proof of \eqref{eq:3.9ok}. In order to prove \eqref{eq:3.9okbis}, we only have to check that under assumption \eqref{eq:smallnesshighint}, we have
\begin{equation}
\dashint_{B_{2r}(x_0)}|\bfu|^{p^*} \, \mathrm{d}x \leq c \dashint_{B_{2r}(x_0)}(1+|\E\bfu|)^p \,\mathrm{d}x \,.
\label{eq:stimanormu}
\end{equation}
Indeed, using the Poincar\'e inequality and the Sobolev-Korn inequality \eqref{eq:2.9bogel}, we get
\begin{equation*}
\begin{split}
\dashint_{B_{2r}(x_0)}|\bfu|^{p^*} \, \mathrm{d}x & \leq c\dashint_{B_{2r}(x_0)}|\bfu-(\bfu)_{x_0,2r}|^{p^*} \, \mathrm{d}x + c |(\bfu)_{x_0,2r}|^{p^*} \\
& \leq c \left(r^{p}\dashint_{B_{2r}(x_0)}|D\bfu|^{p} \, \mathrm{d}x\right)^\frac{p^*}{p} + M^{p^*} \\
& \leq c \left( r^{p}\dashint_{B_{2r}(x_0)}|\E\bfu|^{p}\,\mathrm{d}x + \left (\dashint_{B_{2r}(x_0)}|{\bfu - (\bfu)_{x_0,2r}}| \, \mathrm{d}x\right)^p\right)^\frac{p^*}{p} + M^{p^*} \\
& \leq c\left( r^{p} \dashint_{B_{2r}(x_0)}|\E\bfu|^{p}\,\mathrm{d}x + 1 \right)^\frac{p^*}{p}+ M^{p^*}\\
& \leq c r^{p^*-n(\frac{p^*}{p}-1)} \|1+|\E\bfu|\|_{L^p(\Omega)}^{{p^*}-p} \dashint_{B_{2r}(x_0)}(1+|\E\bfu|^{p})\,\mathrm{d}x +M^{p^*}\\
& \leq c(n,p,\|1+|\E\bfu|\|_{L^p(\Omega)},M) \dashint_{B_{2r}(x_0)}(1+|\E\bfu|^{p})\,\mathrm{d}x\,.
\end{split}
\end{equation*}
The proof is concluded.
\end{proof}

\section{Decay estimates for excess functionals} \label{sec:decayestimate}

{
Let $x_0\in\Omega$, $B_\varrho(x_0)\subset\Omega$ and $\bm\ell_{x_0,\varrho}$ be the traceless affine function defined in \eqref{eq:affinefunctionl}. We introduce the following \emph{Campanato-type} excess functionals, measuring the oscillations of $\E\bfu$:
\begin{equation}
\begin{split}
\Phi(x_0,\varrho): & =\dashint_{B_\varrho(x_0)} H_{1+|(\E\bfu)_{x_0,\varrho}|} (x_0, |\E\bfu-\E\bm\ell_{x_0,\varrho}|)\,\mathrm{d}x \\
 & =\dashint_{B_\varrho(x_0)} H_{1+|(\E\bfu)_{x_0,\varrho}|} (x_0, |\E\bfu-(\E\bfu)_{x_0,\varrho}|)\,\mathrm{d}x  \,,
\end{split}
\label{eq:excess1}
\end{equation}
where $H_{1+|(\E\bfu)_{x_0,\varrho}|} (x_0,t)$ denotes the shifted $N$ function $H_a$ with shift $a:=1+|(\E\bfu)_{x_0,\varrho}|$. }

\subsection{Caccioppoli-type estimates} \label{sec:caccioppolitype}

{The first key tool is} the following ``conditioned''  Caccioppoli type estimate for $\bfu-\bm\ell_{x_0,r}$, {which holds under suitable smallness assumptions, see \eqref{eq:5.6acming} and \eqref{eq:smallness} below. }

\begin{lemma}\label{lem:lemma4.1}
Let $B_{\varrho}(x_0)\subset\subset\Omega$, with $\varrho\leq1$. Assume that
\begin{equation}
\varrho^\frac{p}{2} (|D\bfu|^p)_{x_0,\varrho} + |(\bfu)_{x_0,\varrho}| \leq M\,,
\label{eq:5.6acming}
\end{equation}
and
\begin{equation}
\Phi(x_0,\varrho)
\leq c H(x_0,1+|(\E\bfu)_{x_0,\varrho}|)\,.
\label{eq:smallness}
\end{equation}
Then we have
\begin{equation}
\begin{split}
& \dashint_{B_{\varrho/2}(x_0)} H_{1+|(\E\bfu)_{x_0,\varrho}|}(x_0,|\E\bfu-(\E\bfu)_{x_0,\varrho}|)\,\mathrm{d}x \\
& \,\,\,\,\,\,\,\,\,\,\,\,\,\,\,\,\,\, \leq c \dashint_{B_{\varrho}(x_0)} H_{1+|(\E\bfu)_{x_0,\varrho}|}\left(x_0,\frac{|{\bf u}-\bm\ell_{x_0,\varrho}|}{\varrho}\right)\,\mathrm{d}x + c \varrho^{\bar{\gamma}}\,H(x_0,1+|(\E\bfu)_{x_0,\varrho}|)\,
\end{split}
\label{eq:caccioppoliII}
\end{equation} 
for some constant $c>0$ depending on data, and $\bar{\gamma}:=\min\{\varepsilon_0, 1-\frac{1}{\beta+1}\}$. 
\end{lemma}

\begin{proof}
We follow the argument of \cite[Lemma~4.1]{OKJFA18}. 
We 
use the shorthands $\bm \ell_{\varrho}$ and $(\E\bfu)_{\varrho}$  for $\bm \ell_{x_0,\varrho}$ and $(\E\bfu)_{x_0,\varrho}$, respectively. 
We consider a cut-off function $\eta\in C_0^\infty(B_{\varrho};[0,1])$ such that $\eta\equiv1$ on $B_{\varrho/2}$ and $|D\eta|\leq \frac{c(n)}{\varrho}$. Correspondingly, we define the function $\bm\psi:=\eta^q({\bf u} - \bm \ell_{\varrho})\in W^{1,p}(B_{\varrho};\R^n)$. Since
\begin{equation*}
{\rm div}\, \bm\psi =  D(\eta^q)\cdot ({\bf u} - \bm \ell_{\varrho}) + \eta^q {\rm div}\,({\bf u} - \bm \ell_{\varrho}) =  q\eta^{q-1}D\eta\cdot ({\bf u} - \bm \ell_{\varrho})\,,
\end{equation*}
where we used that ${\rm div}\,\bfu=0$ and $\bm \ell_{\varrho}\in\mathcal{T}(\R^n)$, we conclude that $\bm\psi$ is not a divergence-free vector field. By virtue of Lemma~\ref{lem:bogowski} we can find ${\bf w}\in W^{1,q}_0(B_{\varrho};\R^n)$ such that, setting
\begin{equation*}
\bm\varphi:=\bm\psi-{\bf w} \,,
\end{equation*}
we have ${\rm div}\,\bm\varphi=0$. 
Taking $\bm\varphi$ as a test function in \eqref{eq:1.1ok} we get
\begin{equation}
\begin{split}
\dashint_{B_{\varrho}} \eta^q {\bf a}(x,\E\bfu):\E({\bf u} - \bm \ell_{\varrho})\,\mathrm{d}x & = -q\dashint_{B_{\varrho}}\eta^{q-1} {\bf a}(x,\E\bfu):(D \eta\odot({\bf u} - \bm \ell_{\varrho})-\E{\bf w})\,\mathrm{d}x \\
& \,\,\,\,\,\,\,\, +\dashint_{B_{\varrho}} \bfu[D \bfu][\eta^q({\bf u} - \bm \ell_{\varrho})-{\bf w}]\,\mathrm{d}x \\
& \,\,\,\,\,\,\,\, +  \dashint_{B_{\varrho}}  f[\eta^q({\bf u} - \bm \ell_{\varrho})-{\bf w}]\,\mathrm{d}x\,.
\end{split}
\end{equation}
This and the identity
\begin{equation}
\dashint_{B_{\varrho}}{\bf a}(x_0,(\E\bfu)_{\varrho}):\E\bm\varphi\,\mathrm{d}x=0
\label{eq:4.4ok}
\end{equation}
imply that
\begin{equation}
\begin{split}
J_1:&=\dashint_{B_{\varrho}} \eta^q({\bf a}(x_0,\E\bfu)-{\bf a}(x_0,(\E\bfu)_{\varrho})):(\E{\bf u} -(\E\bfu)_{\varrho})\,\mathrm{d}x \\
 &  =\dashint_{B_{\varrho}}({\bf a}(x_0,\E\bfu)-{\bf a}(x,\E\bfu)):\E{\bm \varphi}\,\mathrm{d}x \\
& \,\,\,\,\,\,  - q\dashint_{B_{\varrho}}\eta^{q-1}({\bf a}(x_0,\E\bfu)-{\bf a}(x_0, (\E\bfu)_{\varrho})):(D\eta\odot({\bf u} - \bm \ell_{\varrho})-\E{\bf w})\,\mathrm{d}x \\
& \,\,\,\,\,\, +\dashint_{B_{\varrho}}\bfu[D \bfu][\eta^q({\bf u} - \bm \ell_{\varrho})-{\bf w}]\,\mathrm{d}x +  \dashint_{B_{\varrho}}  f[\eta^q({\bf u} - \bm \ell_{\varrho})-{\bf w}]\,\mathrm{d}x \\
& = - \dashint_{B_{\varrho}}({\bf a}(x_0, (\E\bfu)_{\varrho})-{\bf a}(x,(\E\bfu)_{\varrho})):\E{\bm \varphi}\,\mathrm{d}x \\
& \,\,\,\,\,\,  - q \dashint_{B_{\varrho}}\eta^{q-1}({\bf a}(x,\E\bfu)-{\bf a}(x,(\E\bfu)_{\varrho})):(({\bf u} - \bm \ell_{\varrho})\odot D\eta-\E{\bf w})\,\mathrm{d}x \\
& \,\,\,\,\,\, +\dashint_{B_{\varrho}} \bfu[D \bfu][\eta^q({\bf u} - \bm \ell_{\varrho})-{\bf w}]\,\mathrm{d}x+ \dashint_{B_{\varrho}}  f[\eta^q({\bf u} - \bm \ell_{\varrho})-{\bf w}]\,\mathrm{d}x \\
& =: J_2+J_3+J_4+ J_5\,.
\end{split}
\label{eq:4.2ok}
\end{equation}
Now, we proceed to estimate each term above separately. With \eqref{eq:1.9ok} and \eqref{(2.6b)} we get
\begin{equation}
\begin{split}
J_1 & \geq c \dashint_{B_{\varrho}} \eta^q H''(x_0,1+|\E\bfu|+|(\E\bfu)_{\varrho}|)|\E{\bf u} - (\E\bfu)_{\varrho}|^2\,\mathrm{d}x \\
     & \geq c \dashint_{B_{\varrho}} \eta^q H_{1+|(\E\bfu)_{\varrho}|}(x_0,|\E{\bf u} - (\E\bfu)_{\varrho}|)\,\mathrm{d}x\,.
\end{split}
\label{eq:stimaJ1}
\end{equation}

To estimate $J_3$ we use the second inequality in \eqref{eq:1.8ok}, \eqref{eq:3.17bog} with $G(t)=H_{1+|(\E\bfu)_{\varrho}|}(x_0,t)$
, Young's inequality for $G, G^*$, \eqref{eq:2.6ok} and we get

\begin{equation}
\begin{split}
|J_3| & \leq c \dashint_{B_{\varrho}}\left(\int_0^1|D_\xi {\bf a}(x_0,\sigma(\E\bfu-(\E\bfu)_{\varrho}) + (\E\bfu)_{\varrho})|\,\mathrm{d}\sigma\right)|\E\bfu-(\E\bfu)_{\varrho}|\left(\frac{|{\bf u} - \bm \ell_{\varrho}|}{\varrho}+|\E{\bf w}|\right)\,\mathrm{d}x \\
& \leq c \dashint_{B_{\varrho}}\frac{H(x_0,1+|(\E\bfu)_{\varrho}|+|\E\bfu-(\E\bfu)_{\varrho}|)}{(1+|(\E\bfu)_{\varrho}|+|\E\bfu-(\E\bfu)_{\varrho}|)^2}|\E\bfu-(\E\bfu)_{\varrho}|\left(\frac{|{\bf u} - \bm \ell_{\varrho}|}{\varrho}+|\E{\bf w}|\right)\,\mathrm{d}x \\
& \leq c \dashint_{B_{\varrho}}{H'_{1+|(\E\bfu)_{\varrho}|}(x_0,|\E\bfu-(\E\bfu)_{\varrho}|)}\left(\frac{|{\bf u} - \bm \ell_{\varrho}|}{\varrho}+|\E{\bf w}|\right)\,\mathrm{d}x \\
& \leq c \dashint_{B_{\varrho}}\left(H^*_{1+|(\E\bfu)_{\varrho}|}\left(x_0,{H'_{1+|(\E\bfu)_{\varrho}|}(x_0,|\E\bfu-(\E\bfu)_{\varrho}|)}\right)+ H_{1+|(\E\bfu)_{\varrho}|}\left(x_0,\frac{|{\bf u} - \bm \ell_{\varrho}|}{\varrho}+|\E{\bf w}|\right)\right)\,\mathrm{d}x \\
& \leq c \dashint_{B_{\varrho}}\left(H_{1+|(\E\bfu)_{\varrho}|}(x_0,|\E\bfu-(\E\bfu)_{\varrho}|)+H_{1+|(\E\bfu)_{\varrho}|}\left(x_0,\frac{|{\bf u} - \bm \ell_{\varrho}|}{\varrho}\right)+H_{1+|(\E\bfu)_{\varrho}|}(x_0,|\E{\bf w}|)\right)\,\mathrm{d}x \\
& \leq c \dashint_{B_{\varrho}}H_{1+|(\E\bfu)_{\varrho}|}(x_0,|\E\bfu-(\E\bfu)_{\varrho}|)\,\mathrm{d}x+ c \dashint_{B_{\varrho}} H_{1+|(\E\bfu)_{\varrho}|}\left(x_0,\frac{|{\bf u} - \bm \ell_{\varrho}|}{\varrho}\right)\,\mathrm{d}x\,.
\end{split}
\label{eq:stimaJ3}
\end{equation}

We now treat separately the two cases where the inequality
\begin{equation}
\inf_{x\in B_{\varrho}} \mu(x) \leq [\mu]_{C^\alpha}\varrho^{\alpha-\varepsilon_0}
\label{eq:4.5ok}
\end{equation}
holds true or not. Note that from \eqref{eq:4.5ok}, for every $x\in B_{\varrho}$ we get
\begin{equation}
\mu(x) \leq  [\mu]_{C^\alpha} \varrho^\alpha + [\mu]_{C^\alpha}\varrho^{\alpha-\varepsilon_0} \leq 2[\mu]_{C^\alpha} \varrho^{\alpha-\varepsilon_0}\,.
\label{eq:4.5okbis}
\end{equation}
\\
{\bf Step 1: $p$-phase.} Let $\mu$ comply with \eqref{eq:4.5ok}. Then, using H\"older's inequality, \eqref{eq:pq}, the definition of $\epsilon_0$ and Lemma~\ref{lem:lemma3.6ok} we get
\begin{equation}
\begin{split}
\varrho^{\alpha-\epsilon_0}(1+|(\E\bfu)_{\varrho}|)^{q-p} & \leq \varrho^{\alpha-\epsilon_0} \left(\dashint_{B_{\varrho}}[1+|\E\bfu|]^{ps_0}\,\mathrm{d}x\right)^{\frac{q-p}{ps_0}} \\
& \leq \varrho^{\alpha-\frac{(q-p)n}{p}+\varepsilon_0} \|1+|\E\bfu|\|_{L^{ps_0}(B_{\varrho})}^{q-p} \leq c \varrho^{\varepsilon_0}\,.
\end{split}
\label{eq:4.6ok}
\end{equation}
From this and \eqref{eq:4.5okbis} we deduce that
\begin{equation}
\begin{split}
H(x_0,1+|(\E\bfu)_{\varrho}|) & = (1+\mu(x_0)(1+|(\E\bfu)_{\varrho}|)^{q-p})(1+|(\E\bfu)_{\varrho}|)^p \\
& \leq c(1+\varrho^{\alpha-\varepsilon_0}(1+|(\E\bfu)_{\varrho}|)^{q-p})(1+|(\E\bfu)_{\varrho}|)^p \\
& \leq c(1+\varrho^{\varepsilon_0})(1+|(\E\bfu)_{\varrho}|)^p \leq c(1+|(\E\bfu)_{\varrho}|)^p\,.
\end{split}
\label{eq:4.7ok}
\end{equation}

As for $J_2$, from \eqref{eq:1.11ok}, \eqref{eq:4.6ok}, \eqref{eq:4.7ok}, and using Young's inequality \eqref{eq:2.5ok} for $H_{1+|(\E\bfu)_{\varrho}|}(x_0,t)$ and its conjugate, \eqref{eq:6.23dieett},  
and \eqref{eq:3.17bogtris},
we obtain
\begin{equation}
\begin{split}
|J_2| & \leq 
 \dashint_{B_{\varrho}} c\varrho^\alpha (1+|(\E\bfu)_{\varrho}|)^{q-1}|\E{\bm \varphi}|\,\mathrm{d}x \\
& \leq c \dashint_{B_{\varrho}}\varrho^{\alpha-\epsilon_0}(1+|(\E\bfu)_{\varrho}|)^{q-p}(1+|(\E\bfu)_{\varrho}|)^{p-1}|\E{\bm \varphi}|\,\mathrm{d}x \\
& \leq c  \dashint_{B_{\varrho}}\varrho^{\varepsilon_0} H'(x_0, 1+|(\E\bfu)_{\varrho}|)|\E{\bm \varphi}|\,\mathrm{d}x \\ 
& \leq \frac{1}{4}\dashint_{B_{\varrho}}\eta^q H_{1+|(\E\bfu)_{\varrho}|}(x_0,|\E\bfu-(\E\bfu)_{\varrho}|)\,\mathrm{d}x + c \dashint_{B_{\varrho}}H_{1+|(\E\bfu)_{\varrho}|}\left(x_0,\frac{|\bfu-{\bm \ell}_{\varrho}|}{\varrho}\right)\,\mathrm{d}x \\
&\,\,\,\,\,\,  + c \varrho^{\varepsilon_0} H(x_0, 1+|(\E\bfu)_{\varrho}|)\,.
\end{split}
\label{eq:stimaJ2}
\end{equation}

Plugging the estimates \eqref{eq:stimaJ1}, \eqref{eq:stimaJ2} and \eqref{eq:stimaJ3} 
into \eqref{eq:4.2ok} and reabsorbing some terms we obtain \eqref{eq:caccioppoliII}. \\
\\
{\bf Step 2: $(p,q)$-phase.} If \eqref{eq:4.5ok} does not hold, then $\mu(x_0)\geq\inf_{x\in B_{\varrho}} \mu(x) > [\mu]_{C^\alpha}\varrho^{\alpha-\varepsilon_0}$, so that
\begin{equation}
\varrho^\alpha \leq c \varrho^{\varepsilon_0}\mu(x_0)\,.
\label{eq:stimamu1}
\end{equation} 

Now, arguing as for Step~1 and using Young's inequality \eqref{eq:2.5ok} for $H_{1+|(\E\bfu)_{\varrho}|}(x_0,t)$ and its conjugate,  and \eqref{eq:6.23dieett}, we have
\begin{equation}
\begin{split}
|J_2| & 
 \leq c \dashint_{B_{\varrho}}\varrho^{\varepsilon_0}H'(x_0,1+|(\E\bfu)_{\varrho}|)|\E{\bm \varphi}|\,\mathrm{d}x \\
& \leq \frac{1}{4}\dashint_{B_{\varrho}}\eta^q H_{1+|(\E\bfu)_{\varrho}|}(x_0, |\E\bfu-(\E\bfu)_{\varrho}|)\,\mathrm{d}x + c \dashint_{B_{\varrho}}H_{1+|(\E\bfu)_{\varrho}|}\left(x_0, \frac{|\bfu-{\bm \ell}_{\varrho}|}{\varrho}\right)\,\mathrm{d}x \\
&\,\,\,\,\,\,  + c \varrho^{\varepsilon_0}H(x_0, 1+|(\E\bfu)_{\varrho}|)\,. 
\end{split}
\end{equation}

Now, we set $\gamma:=\left(\frac{p^*}{2}\right)'$.
Note that if $p>\frac{3n}{n+2}$, then $\gamma<p$. Using 
H\"older's inequality we have
\begin{equation*}
\begin{split}
|J_4| & \leq \dashint_{B_{\varrho}} |\bfu| |D\bfu||\eta^q({\bf u} - \bm \ell_{\varrho})-{\bf w}|\,\mathrm{d}x \\
 & \leq \left(\dashint_{B_{\varrho}} |\bfu|^{p^*}\,\mathrm{d}x\right)^\frac{1}{p^*}\left(\dashint_{B_{\varrho}} |D\bfu|^{\gamma}\,\mathrm{d}x\right)^\frac{1}{\gamma}\left(\dashint_{B_{\varrho}} (|\bfu - \bm\ell_\varrho|^{p^*}+|{\bf w}|^{p^*})\,\mathrm{d}x\right)^\frac{1}{p^*} \\
&  =: I_1\cdot I_2\cdot I_3 \,.
\end{split}
\end{equation*}
Estimate of $I_1$: using Poincar\'e inequality and \eqref{eq:5.6acming}, we get

\begin{equation}
|I_1| \leq \left(\dashint_{B_{\varrho}} |\bfu-(\bfu)_\varrho|^{p^*}\,\mathrm{d}x\right)^\frac{1}{p^*} + |(\bfu)_\varrho| \leq \varrho(|D\bfu|^p)_\varrho + M \leq 2M\,.
\end{equation}
Estimate of $I_2$: with the Sobolev-Korn's inequality \eqref{eq:2.9bogel} and \eqref{eq:5.6acming} we obtain

\begin{equation}
\begin{split}
|I_2| & \leq c \left[\left(\dashint_{B_{\varrho}} |\E\bfu|^{\gamma}\,\mathrm{d}x\right)^\frac{1}{\gamma} + \dashint_{B_{\varrho}} \left|\frac{\bfu-(\bfu)_\varrho}{\varrho}\right|\,\mathrm{d}x\right] \\
& \leq c \left[\left(\dashint_{B_{\varrho}} |\E\bfu|^{\gamma}\,\mathrm{d}x\right)^\frac{1}{\gamma} + C_M \right]\,.
\end{split}
\end{equation}
Now, using Lemma~\ref{lem:changeshift} and \eqref{eq:smallness},
\begin{equation}
\begin{split}
&\dashint_{B_{\varrho}} H(x_0, 1+|\E\bfu|)\,\mathrm{d}x \\
& \leq cc_\eta \dashint_{B_{\varrho}} H_{1+|(\E\bfu)_\varrho|}(x_0, |\E\bfu-(\E\bfu)_\varrho|)\,\mathrm{d}x + (\eta+c) H(x_0, 1+|(\E\bfu)_\varrho|) \\
& \leq c H(x_0, 1+|(\E\bfu)_\varrho|)\,,
\end{split}
\end{equation}
whence, from Jensen's inequality for the function $\Psi(t):=t^\frac{p}{\gamma}+\mu(x_0)t^\frac{q}{\gamma}$ (convex, since $\gamma<p\leq q$) we obtain
\begin{equation}
\begin{split}
\dashint_{B_{\varrho}} (1+|\E\bfu|)^{\gamma}\,\mathrm{d}x & \leq \Psi^{-1}\left(\dashint_{B_{\varrho}} H(x_0, 1+|\E\bfu|)\,\mathrm{d}x\right) \\
& \leq c \Psi^{-1}\left(H(x_0, 1+|(\E\bfu)_\varrho|)\right) \\
& \leq c(1+|(\E\bfu)_\varrho|)^\gamma\,.
\end{split}
\label{eq:otherestimate}
\end{equation}
Thus,
\begin{equation}
\begin{split}
|I_2| \leq c \left[1+|(\E\bfu)_\varrho|+ C_M \right]\,.
\end{split}
\end{equation}
Estimate of $I_3$: we use Poincar\'e's inequality, Lemma~\ref{lem:bogowski} and Lemma~\ref{lem:kornorlicz}, and we get

\begin{equation}
\begin{split}
|I_3| \leq \left(\dashint_{B_{\varrho}} (|\bfu - \bm\ell_\varrho|^{p^*}+\varrho^{p^*}|D {\bf w}|^{p})\,\mathrm{d}x\right)^\frac{1}{p^*} & \leq c \left(\dashint_{B_{\varrho}} |\bfu - \bm\ell_\varrho|^{p^*}\,\mathrm{d}x\right)^\frac{1}{p^*} \\
& \leq c \varrho\left(\dashint_{B_{\varrho}} |D\bfu - (D\bfu)_\varrho|^{p}\,\mathrm{d}x\right)^\frac{1}{p}\\
& \leq c c_{\rm Korn} \varrho\left(\dashint_{B_{\varrho}} |\E\bfu - (\E\bfu)_\varrho|^{p}\,\mathrm{d}x\right)^\frac{1}{p}\,.
\end{split}
\end{equation}
Now, using the change of shift formula, \eqref{eq:smallness} and an analogous argument as for \eqref{eq:otherestimate} with the convex function $\widetilde{\Psi}(t):= t +\mu(x_0)t^\frac{q}{p}$, we obtain
\begin{equation}
\left(\dashint_{B_{\varrho}} |\E\bfu - (\E\bfu)_\varrho|^{p}\,\mathrm{d}x\right)^\frac{1}{p} \leq c (1+|(\E\bfu)_\varrho|)\,.
\end{equation}
Collecting the previous estimates, we finally get
\begin{equation*}
\begin{split}
\left|\dashint_{B_{\varrho}} \bfu[D\bfu][\eta^q({\bf u} - \bm \ell_{\varrho})-{\bf w}]\,\mathrm{d}x\right|  & \leq c \varrho (1+|(\E\bfu)_\varrho|)^2  \\
& \leq c  \varrho H(x_0, 1+|(\E\bfu)_\varrho|)\,.
\end{split}
\end{equation*}

We are left to estimate $J_5$. Using H\"older's inequality, $p^*>\frac{n}{n-1}$ and the estimate of $I_3$ we have
\begin{equation*}
\begin{split}
|J_5| & \leq c \left(\dashint_{B_\varrho} |f|^n\,\mathrm{d}x\right)^\frac{1}{n} \left(\dashint_{B_{\varrho}} (|\bfu - \bm\ell_\varrho|^{\frac{n}{n-1}}+|{\bf w}|^{\frac{n}{n-1}})\,\mathrm{d}x\right)^\frac{n-1}{n} \\
& \leq c \varrho^{-\frac{1}{\beta+1}}\|f\|_{L^{n(1+\beta)}(\Omega)} \left(\dashint_{B_{\varrho}} (|\bfu - \bm\ell_\varrho|^{p^*}+|{\bf w}|^{p^*})\,\mathrm{d}x\right)^\frac{1}{p^*} \\
& \leq c \varrho^{1-\frac{1}{\beta+1}} (1+|(\E\bfu)_\varrho|) \\
& \leq c \varrho^{1-\frac{1}{\beta+1}}\,H(x_0, 1+|(\E\bfu)_\varrho|)\,.
\end{split}
\end{equation*}
The proof of \eqref{eq:caccioppoliII} is then concluded. 
\end{proof}

The following result can be obtained combining the argument of \cite[Theorem~3.2]{OKJFA18} with Lemma~\ref{lem:kornorlicz}.
\begin{lemma}
There exist $0<\sigma_1<1$ depending on $n,p,q$ such that for any $B_\varrho(x_0)\subset\subset\Omega$ we have
\begin{equation*}
\begin{split}
&\dashint_{B_{\varrho}(x_0)}   H_{1+|(\E\bfu)_{x_0,\varrho}|}\left(x_0,\frac{|{\bf u}-\bm\ell_{x_0,\varrho}|}{\varrho}\right)\,\mathrm{d}x \\
& \leq c \left(1+ [\mu]_{C^\alpha}\|\E\bfu\|^{q-p}_{L^p(\Omega)}\varrho^{\alpha-\frac{(q-p)n}{p}}\right) \left(\dashint_{B_{\varrho}(x_0)} \left[H_{1+|(\E\bfu)_{x_0,\varrho}|}(x_0, |\E\bfu-(\E\bfu)_{x_0,\varrho}|)\right]^{\sigma_1}\,\mathrm{d}x \right)^{\frac{1}{\sigma_1}}
\end{split}
\end{equation*}
for some constant $c=c(n,p,q)>1$.
\label{lem:thm3.2ok}
\end{lemma}
\begin{proof}
Let $a\geq0$ and denote by $H_a(x_0,t)$ the shifted $N$-function of $H(x_0,t)$ defined as in \eqref{eq:phi_shifted}. Note that, by the definition of $H$ and \eqref{(2.6b)}, we have
\begin{equation}
\begin{split}
H_a(x_0,t) & \sim (a+t)^{p-2} t^2 + \mu(x_0) (a+t)^{q-2} t^2 \\
& \,\,\,\, \,\,\,\, =: G_{p,a}(t) + \mu(x_0) G_{q,a}(t)\,,
\end{split}
\end{equation}
where the $N$-functions $G_{p,a}$ and $G_{q,a}$ are equivalent to the corresponding shifted $N$-functions of $t^p$ and $t^q$, respectively. 
Note that
\begin{equation}
G_{p,a}(t) \sim a^p + t^p \sim (a+t)^p\,, \quad G_{q,a}(t) \sim a^q + t^q \sim (a+t)^q\,.
\end{equation}

We first assume that
\begin{equation}
\sup_{x\in B_\varrho} \mu(x) \leq 4 [\mu]_{C^\alpha} \varrho^\alpha\,.
\label{eq:(3.4)ok2018}
\end{equation}
To enlighten the notation, we set
\begin{equation*}
{\bf f}_{x_0,\varrho}:= {\bf u}-\bm\ell_{x_0,\varrho}\,.
\end{equation*}
Note that $({\bf f}_{x_0,\varrho})_{x_0,\varrho}={\bf 0}$.
Using the Sobolev-Poincar\'e inequality, together with \eqref{eq:(3.4)ok2018}, we get 
\begin{equation}
\begin{split}
\dashint_{B_\varrho}\left(a+\frac{|{\bf f}_{x_0,\varrho}|}{\varrho}\right)^p \,\mathrm{d}x & \leq c  \left(\dashint_{B_\varrho} \left(a+|D{\bf f}_{x_0,\varrho}|\right)^\frac{np}{n+p}\,\mathrm{d}x\right)^\frac{n+p}{n}\,, \\
\dashint_{B_\varrho} \mu(x_0)\left(a+\frac{|{\bf f}_{x_0,\varrho}|}{\varrho}\right)^q \,\mathrm{d}x & \leq c 4 [\mu]_{C^\alpha} \varrho^\alpha  \left(\dashint_{B_\varrho} (a+|D{\bf f}_{x_0,\varrho}|)^\frac{nq}{n+q}\,\mathrm{d}x\right)^\frac{n+q}{n}\,. \label{eq:stimaaa}
\end{split}
\end{equation}
Now, invoking the Korn's inequality of Lemma~\ref{lem:kornorlicz}, we have
\begin{equation*}
\begin{split}
\dashint_{B_\varrho}\left(a+\frac{|{\bf f}_{x_0,\varrho}|}{\varrho}\right)^p \,\mathrm{d}x & \leq c_{\rm Korn} \left(\dashint_{B_\varrho} \left(a+|\E{\bf f}_{x_0,\varrho}|\right)^\frac{np}{n+p}\,\mathrm{d}x\right)^\frac{n+p}{n}\,, \\
\dashint_{B_\varrho} \mu(x_0)\left(a+\frac{|{\bf f}_{x_0,\varrho}|}{\varrho}\right)^q \,\mathrm{d}x & \leq c c_{\rm Korn} 4 [\mu]_{C^\alpha} \varrho^\alpha \left(\dashint_{B_\varrho} (a+|\E{\bf f}_{x_0,\varrho}|)^\frac{nq}{n+q}\,\mathrm{d}x\right)^\frac{n+q}{n}\,.
\end{split}
\end{equation*}
By H\"older's inequality, we get
\begin{equation}
\begin{split}
\left(\dashint_{B_\varrho} (a+|\E{\bf f}_{x_0,\varrho}|)^\frac{nq}{n+q}\,\mathrm{d}x\right)^\frac{n+q}{n} & \leq \left(\dashint_{B_\varrho} (a+|\E{\bf f}_{x_0,\varrho}|)^p\,\mathrm{d}x\right)^\frac{q-p}{p} \left(\dashint_{B_\varrho} (a+|\E{\bf f}_{x_0,\varrho}|)^\frac{nq}{n+q}\,\mathrm{d}x\right)^\frac{p(n+q)}{n} \\
& \leq \varrho^{-\frac{n(q-p)}{p}} \|a+|\E\bfu|\|_{L^p(\Omega)}^{q-p} \left(\dashint_{B_\varrho} (a+|\E{\bf f}_{x_0,\varrho}|)^{p\sigma_1}\,\mathrm{d}x\right)^\frac{1}{\sigma_1}\,,
\end{split}
\end{equation}
where we have set $\sigma_1:= \frac{nq}{p(n+q)}<1$. Combining with \eqref{eq:stimaaa}, we then obtain
\begin{equation}
\begin{split}
\dashint_{B_\varrho} & \left[\left(a+\frac{|{\bf f}_{x_0,\varrho}|}{\varrho}\right)^p \,\mathrm{d}x+ \mu(x_0)\left(a+\frac{|{\bf f}_{x_0,\varrho}|}{\varrho}\right)^q \right]\,\mathrm{d}x  \\
& \leq c \left[1+4 [\mu]_{C^\alpha} \varrho^{\alpha-\frac{n(q-p)}{p}} \|a+|\E\bfu|\|_{L^p(\Omega)}^{q-p}\right] \left(\dashint_{B_\varrho} (a+|\E{\bf f}_{x_0,\varrho}|)^{p\sigma_1}\,\mathrm{d}x\right)^\frac{1}{\sigma_1}\,.
\end{split}
\end{equation}
{If instead \eqref{eq:(3.4)ok2018} does not hold, we can apply Theorem~\ref{thm:sob-poincare} and Lemma~\ref{lem:kornorlicz} with $G(t):=H_a(x_0,t)$. The proof is concluded.}
\end{proof}

{Now, we are in position to establish a higher integrability result for $H_{1+|(\E\bfu)_{x_0,\varrho}|}(x_0, |\E\bfu-(\E\bfu)_{x_0,\varrho}|)$. The result follows from Lemma~\ref{lem:thm3.2ok} and Lemma~\ref{lem:lemma4.1} as a consequence of Gehring's lemma with increasing supports (Lemma~\ref{lem:gehring}): 
\begin{corollary}\label{corollary3.2}
There exist a  constant $c_{\rm high}=c_{\rm high}(data)>0$ and $\sigma>1$ such that 
\begin{equation}
\begin{split}
&\left(\dashint_{B_{\varrho/2}(x_0)} \left[H_{1+|(\E\bfu)_{x_0,\varrho}|}(x_0, |\E\bfu-(\E\bfu)_{x_0,\varrho}|)\right]^\sigma\,\mathrm{d}x \right)^{\frac{1}{\sigma}} \\
&\leq c_{\rm high} \dashint_{B_{\varrho}(x_0)} H_{1+|(\E\bfu)_{x_0,\varrho}|}(x_0, |\E\bfu-(\E\bfu)_{x_0,\varrho}|)\,\mathrm{d}x + c_{\rm high} \varrho^{\widetilde{\gamma}}\, H(x_0, 1+|(\E\bfu)_{x_0,\varrho}|)\,.
\end{split}
\label{eq:caccioppoliIbis}
\end{equation} 
\end{corollary}}

\subsection{Almost $\mathcal{A}$-Stokes functions} \label{sec:almostAStokes}

{We can start with the linearization procedure for system \eqref{eq:1.1ok}.} Let us set

\begin{equation}\label{operatorA}
\mathcal{A}\colon= \frac{ D_\xi{\bf a}(x_0, (\E\bfu)_{x_0,\varrho})}{ H''(x_0, 1+|(\E\bfu)_{x_0,\varrho}|)} \quad \mbox{ and } \quad {\bf w}:= {\bfu - \bm\ell_{x_0,\varrho}}\,.
\end{equation}
Note that the bilinear form $\mathcal{A}$ satisfies the Legendre-Hadamard condition \eqref{eq:LegHam} by virtue of \eqref{eq:1.8ok}.

We aim to prove that ${\bf w}$ is approximately $\mathcal{A}$-Stokes. {This fact, together with the higher integrability result \eqref{eq:caccioppoliIbis} will allow us to apply the $\mathcal{A}$-Stokes approximation theorem.}

\begin{lemma}
Let $B_{\varrho}(x_0)\subset\subset\Omega$. {Assume that \eqref{eq:5.6acming} and \eqref{eq:smallness} hold, and let $\varrho\leq1$. }
Then there exists a constant $c_{\rm Stokes}>0$ such that
\begin{equation}
\begin{split}
& \left|\dashint_{B_{\varrho/2}(x_0)}\mathcal{A}\E{\bf w}:\E\bm\varphi\,\mathrm{d}x\right| \\
& \leq c_{\rm Stokes} (1+|(\E\bfu)_{x_0,\varrho}|)\mathcal{R}\left(\frac{\Phi(x_0,\varrho)}{H(x_0, 1+|(\E\bfu)_{x_0,\varrho}|)}+\varrho^{\widetilde{\gamma}}\right)\|D\bm\varphi\|_{L^\infty(B_{\varrho/2}(x_0))} 
\end{split}
\label{eq:4.12ok}
\end{equation}
for every $\bm\varphi\in C_{0,{\rm div}}^\infty(B_{\varrho/2}(x_0))$, where
$\mathcal{R}(t):= [\omega(t^\frac{1}{2}) + t ]^\frac{1}{2} \sqrt{t}$.
\label{lem:lemma4.2ok}
\end{lemma}

\begin{proof}
It will suffice to prove \eqref{eq:4.12ok} for any fixed $\bm\varphi\in C_0^\infty(B_{\varrho/2}(x_0))$ with ${\rm div}\,\bm\varphi=0$ such that $\|D\bm\varphi\|_{L^\infty(B_{\varrho/2}(x_0))}\leq1$, since the general case will follow by a standard normalization argument. To enlighten notation, we omit the explicit dependence on $x_0$.

From the definitions of $\mathcal{A}$ and ${\bf w}$ we have
\begin{equation}
\begin{split}
& H''(x_0,1+|(\E\bfu)_\varrho|) \dashint_{B_{\varrho/2}}\mathcal{A}\E{\bf w}:\E\bm\varphi\,\mathrm{d}x  \\
&\,\, = \dashint_{B_{\varrho/2}} D_\xi{\bf a}(x_0, (\E\bfu)_\varrho) \E(\bfu-\bm\ell_\varrho):\E\bm\varphi\,\mathrm{d}x \\
& \,\, = \dashint_{B_{\varrho/2}} \int_0^1 \left[D_\xi{\bf a}(x_0, (\E\bfu)_\varrho) - D_\xi{\bf a}(x_0, (\E\bfu)_\varrho + t \E(\bfu-\bm\ell_\varrho))\right]\E(\bfu-\bm\ell_\varrho) : \E\bm\varphi\,\mathrm{d}t\,\mathrm{d}x \\
& \,\, \,\,\,\, + \dashint_{B_{\varrho/2}} \int_0^1 \left[D_\xi{\bf a}(x_0, (\E\bfu)_\varrho + t \E(\bfu-\bm\ell_\varrho))\right]\E(\bfu-\bm\ell_\varrho) : \E\bm\varphi\,\mathrm{d}t\,\mathrm{d}x \\
& \,\, =: J_1 + J_2\,.
\end{split}
\label{eq:4.13ok}
\end{equation}
We set $G(t):=H(x_0, t)$. 
From \eqref{eq:1.12ok}, \eqref{ineq:phiast_phi_p}, 
\eqref{(2.1celokbis)} and Lemma~\ref{technisch-mu},  we have
\begin{equation*}
\begin{split}
|J_1| & \leq c \dashint_{B_{\varrho/2}}  \int_0^1 \omega\left(\frac{t |\E\bfu-(\E\bfu)_\varrho|}{1+|(\E\bfu)_\varrho|}\right) G''(1+|(\E\bfu)_\varrho|+|(\E\bfu)_\varrho+t(\E\bfu-(\E\bfu)_\varrho)|) \,\mathrm{d}t\,|\E\bfu-(\E\bfu)_\varrho|\,\mathrm{d}x \\
 & \leq c \dashint_{B_{\varrho/2}}  \omega\left(\frac{ |\E\bfu-(\E\bfu)_\varrho|}{1+|(\E\bfu)_\varrho|}\right) \left[\frac{G'(1+|(\E\bfu)_\varrho|+|\E\bfu-(\E\bfu)_\varrho|)}{1+|(\E\bfu)_\varrho|+|\E\bfu-(\E\bfu)_\varrho|}\right]|\E\bfu-(\E\bfu)_\varrho|\,\mathrm{d}x \\
& \leq c \dashint_{B_{\varrho/2}} \mathbbm{1}_E(x)  \omega\left(\frac{ |\E\bfu-(\E\bfu)_\varrho|}{1+|(\E\bfu)_\varrho|}\right) \frac{G'(1+|(\E\bfu)_\varrho|+|\E\bfu-(\E\bfu)_\varrho|)}{1+|(\E\bfu)_\varrho|}|\E\bfu-(\E\bfu)_\varrho|\,\mathrm{d}x \\
& \,\,\,\,\,\, + c \dashint_{B_{\varrho/2}} \mathbbm{1}_F(x)  \omega\left(\frac{ |\E\bfu-(\E\bfu)_\varrho|}{1+|(\E\bfu)_\varrho|}\right) \frac{G'(1+|(\E\bfu)_\varrho|+|\E\bfu-(\E\bfu)_\varrho|)}{1+|(\E\bfu)_\varrho|}|\E\bfu-(\E\bfu)_\varrho|\,\mathrm{d}x \\
& =: J_{1,E} + J_{1,F}\,,
\end{split}
\end{equation*}
where $E:=\{x\in B_{\varrho/2}:\,\, |\E\bfu(x)-(\E\bfu)_\varrho|>1+|(\E\bfu)_\varrho|\}$, and $F:=B_\varrho\backslash E$.

We start with the estimate of $J_{1,E}$. We have
\begin{equation*}
\begin{split}
\frac{|J_{1,E}|}{G'(1+|(\E\bfu)_\varrho|)} 
& \leq c \dashint_{B_{\varrho/2}}  \mathbbm{1}_E(x)  \omega\left(\frac{|\E\bfu-(\E\bfu)_\varrho|}{1+|(\E\bfu)_\varrho|}\right) \frac{G'(1+|(\E\bfu)_\varrho|+|\E\bfu-(\E\bfu)_\varrho|)}{G(1+|(\E\bfu)_\varrho|)}|\E\bfu-(\E\bfu)_\varrho|\,\mathrm{d}x\,.
\end{split}
\end{equation*}
For a.e. $x\in E$, we have
\begin{equation*}
1+|(\E\bfu)_\varrho|+|\E\bfu-(\E\bfu)_\varrho| \leq 2 |\E\bfu-(\E\bfu)_\varrho|\,,
\end{equation*}
and taking into account \eqref{(2.6d)}, 
we obtain
\begin{equation*}
\begin{split}
G'(1+|(\E\bfu)_\varrho|+|\E\bfu-(\E\bfu)_\varrho|) |\E\bfu-(\E\bfu)_\varrho| & \leq c G(|\E\bfu-(\E\bfu)_\varrho|) \\
& \leq c G(1+|(\E\bfu)_\varrho|+|\E\bfu-(\E\bfu)_\varrho|)\\
& \leq c G_{1+|(\E\bfu)_\varrho|}(|\E\bfu-(\E\bfu)_\varrho|)\,.
\end{split}
\end{equation*}
Now, using that $\omega\leq1$, we finally get
\begin{equation*}
\mathbbm{1}_E(x)  \omega\left(\frac{|\E\bfu-(\E\bfu)_\varrho|}{1+|(\E\bfu)_\varrho|}\right) \frac{G'(1+|(\E\bfu)_\varrho|+|\E\bfu-(\E\bfu)_\varrho|)}{G(1+|(\E\bfu)_\varrho|)}|\E\bfu-(\E\bfu)_\varrho| \leq \frac{G_{1+|(\E\bfu)_\varrho|}(|\E\bfu-(\E\bfu)_\varrho|)} {G(1+|(\E\bfu)_\varrho|)}\,,
\end{equation*}
whence
\begin{equation*}
|J_{1,E}| \leq c G'(1+|(\E\bfu)_\varrho|)\dashint_{B_{\varrho/2}} \frac{G_{1+|(\E\bfu)_\varrho|}(|\E\bfu-(\E\bfu)_\varrho|)} {G(1+|(\E\bfu)_\varrho|)}\,\mathrm{d}x\,.
\end{equation*}
For what concerns $J_{1,F}$, arguing as for $J_{1,E}$ we get the preliminary estimate
\begin{equation*}
\frac{|J_{1,F}|}{G'(1+|(\E\bfu)_\varrho|)}\leq c \dashint_{B_{\varrho/2}}  \mathbbm{1}_F(x)  \omega\left(\frac{|\E\bfu-(\E\bfu)_\varrho|}{1+|(\E\bfu)_\varrho|}\right) \frac{G'(1+|(\E\bfu)_\varrho|+|\E\bfu-(\E\bfu)_\varrho|)}{G(1+|(\E\bfu)_\varrho|)}|\E\bfu-(\E\bfu)_\varrho|\,\mathrm{d}x\,.
\end{equation*}
For a.e. $x\in F$, we have
\begin{equation*}
1+|(\E\bfu)_\varrho|+|\E\bfu-(\E\bfu)_\varrho| < 2(1+|(\E\bfu)_\varrho|)\,,
\end{equation*}
and with \eqref{ineq:phiast_phi_p} and since $G'$ is increasing, we can write  
\begin{equation*}
|J_{1,F}|\leq c G'(1+|(\E\bfu)_\varrho|)\dashint_{B_{\varrho/2}}  \mathbbm{1}_F(x)  \omega\left(\frac{|\E\bfu-(\E\bfu)_\varrho|}{1+|(\E\bfu)_\varrho|}\right) \frac{|\E\bfu-(\E\bfu)_\varrho|}{1+|(\E\bfu)_\varrho|}\,\mathrm{d}x\,.
\end{equation*}
Using the definition of $F$ again, the monotonicity of $G'$ and \eqref{ineq:phiast_phi_p}, 
for a.e. $x\in F$ we have
\begin{equation*}
\begin{split}
& \omega\left(\frac{|\E\bfu-(\E\bfu)_\varrho|}{1+|(\E\bfu)_\varrho|}\right) \frac{|\E\bfu-(\E\bfu)_\varrho|}{1+|(\E\bfu)_\varrho|} \\
&  \leq \omega\left(\left[\frac{G'(1+|(\E\bfu)_\varrho|)|\E\bfu-(\E\bfu)_\varrho|^2}{G'(1+|(\E\bfu)_\varrho|)(1+|(\E\bfu)_\varrho|)^2}\right]^\frac{1}{2}\right) \left[\frac{G'(1+|(\E\bfu)_\varrho|)|\E\bfu-(\E\bfu)_\varrho|^2}{G'(1+|(\E\bfu)_\varrho|)(1+|(\E\bfu)_\varrho|)^2} \right]^\frac{1}{2}\\
&  \leq \omega\left(\left[\frac{G'(1+|(\E\bfu)_\varrho|+|\E\bfu-(\E\bfu)_\varrho|)|\E\bfu-(\E\bfu)_\varrho|^2}{G(1+|(\E\bfu)_\varrho|)(1+|(\E\bfu)_\varrho|+|\E\bfu-(\E\bfu)_\varrho|)}\right]^\frac{1}{2}\right) \left[\frac{G'(1+|(\E\bfu)_\varrho|+|\E\bfu-(\E\bfu)_\varrho|)|\E\bfu-(\E\bfu)_\varrho|^2}{G(1+|(\E\bfu)_\varrho|)(1+|(\E\bfu)_\varrho|+|\E\bfu-(\E\bfu)_\varrho|)}\right]^\frac{1}{2}\,.
\end{split}
\end{equation*}
Using \eqref{(2.6b)} we then get
\begin{equation*}
\begin{split}
& \omega\left(\frac{|\E\bfu-(\E\bfu)_\varrho|}{1+|(\E\bfu)_\varrho|}\right) \frac{|\E\bfu-(\E\bfu)_\varrho|}{1+|(\E\bfu)_\varrho|} \\
&  \leq \omega\left(\left[\frac{G_{1+|(\E\bfu)_\varrho|}(|\E\bfu-(\E\bfu)_\varrho|)}{G(1+|(\E\bfu)_\varrho|)}\right]^\frac{1}{2}\right) \left[\frac{G_{1+|(\E\bfu)_\varrho|}(|\E\bfu-(\E\bfu)_\varrho|)}{G(1+|(\E\bfu)_\varrho|)}\right]^\frac{1}{2}\,.
\end{split}
\end{equation*}
{Combining the previous estimates and using H\"older's inequality, the fact that $\omega\leq1$ and Jensen's inequality for the concave function $\omega(t^\frac{1}{2})$, we get
\begin{equation*}
\frac{|J_{1,F}|}{G'(1+|(\E\bfu)_\varrho|)} \leq c \left\{ \left[\omega\left(\left[\dashint_{B_{\varrho/2}}\frac{G_{1+|(\E\bfu)_\varrho|}(|\E\bfu-(\E\bfu)_\varrho|)} {G(1+|(\E\bfu)_\varrho|)}\,\mathrm{d}x \right]^\frac{1}{2}\right) \right]^\frac{1}{2} \left( \dashint_{B_{\varrho/2}} \frac{G_{1+|(\E\bfu)_\varrho|}(|\E\bfu-(\E\bfu)_\varrho|)} {G(1+|(\E\bfu)_\varrho|)}\,\mathrm{d}x  \right)^\frac{1}{2}\right\}\,.
\end{equation*}
Collecting the estimates for $J_{1,E}$ and $J_{1,F}$, we then infer
\begin{equation}
\begin{split}
& \frac{|J_1|}{G''(1+|(\E\bfu)_\varrho|)(1+|(\E\bfu)_\varrho|)} \\
& \leq c \left[\frac{\Phi(\varrho)}{G(1+|(\E\bfu)_\varrho|)} + \sqrt{\omega\left(\left[\frac{\Phi(\varrho)}{G(1+|(\E\bfu)_\varrho|)}\right]^\frac{1}{2}\right)}\sqrt{\frac{\Phi(\varrho)}{G(1+|(\E\bfu)_\varrho|)}}\right] \,. 
\end{split}
\label{eq:4.14ok}
\end{equation}}

In order to estimate $J_2$, we use \eqref{eq:4.4ok}, the definition of weak solution, \eqref{eq:1.11ok} and we preliminarly obtain
\begin{equation}
\begin{split}
|J_2| & = \left|\dashint_{B_{\varrho/2}}  \left[{\bf a}(x_0, \E\bfu)-{\bf a}(x_0, (\E\bfu)_\varrho)\right]: \E\bm\varphi\,\mathrm{d}x \right| \\
 & = \left|\dashint_{B_{\varrho/2}}  \left[{\bf a}(x_0, \E\bfu)-{\bf a}(x, \E\bfu)\right]: \E\bm\varphi\,\mathrm{d}x + \dashint_{B_{\varrho/2}} [\bfu[D\bfu]+f]\bm\varphi\,\mathrm{d}x \right| \\
& \leq  c \dashint_{B_{\varrho/2}}|\mu(x)-\mu(x_0)|(1+|\E\bfu|)^{q-1}\,\mathrm{d}x + \dashint_{B_{\varrho/2}}|\bfu||D\bfu||\bm\varphi|\,\mathrm{d}x + \dashint_{B_{\varrho/2}}|f||\bm\varphi|\,\mathrm{d}x \\
& =: J_{2,1}+J_{2,2}+J_{2,3}.
\end{split}
\label{eq:4.15ok}
\end{equation}

Now, we have to distinguish between two cases, depending on whether condition \eqref{eq:4.5ok} is satisfied or not. \\
\\
{\bf Step 1: $p$-phase.} Under assumption \eqref{eq:4.5ok}, we have 
\begin{equation}
\begin{split}
|J_{2,1}| & \leq c \varrho^{\frac{\alpha-\varepsilon_0}{q}}\left(\dashint_{B_{\varrho/2}}|\mu(x)|^{\frac{q-1}{q}}(1+|\E\bfu|)^{q-1}\,\mathrm{d}x +  \dashint_{B_{\varrho/2}}|\mu(x_0)|^{\frac{q-1}{q}}(1+|\E\bfu|)^{q-1}\,\mathrm{d}x\right) \\
& \leq c \varrho^{\frac{\alpha-\varepsilon_0}{q}}\left[\left(\dashint_{B_{\varrho/2}}\mu(x)(1+|\E\bfu|)^{q}\,\mathrm{d}x\right)^\frac{q-1}{q} + \left(\dashint_{B_{\varrho/2}}\mu(x_0)(1+|\E\bfu|)^{q}\,\mathrm{d}x\right)^\frac{q-1}{q}\right]\,.
\end{split}
\label{eq:stimaJ2bis}
\end{equation}
An analogous argument as for \eqref{eq:4.6ok}-\eqref{eq:4.7ok} shows that
\begin{equation*}
\begin{split}
H(x_0, 1+|(\E\bfu)_\varrho|) & \leq c(1+\varrho^{\alpha-\varepsilon_0}(1+|(\E\bfu)_\varrho|)^{q-p}) (1+|(\E\bfu)_\varrho|)^p \\
& \leq c (1+|(\E\bfu)_\varrho|)^p\,,
\end{split}
\end{equation*}
whence, using the change of shift formula \eqref{(5.4diekreu)} and \eqref{eq:smallness}, we obtain
\begin{equation*}
\begin{split}
\dashint_{B_\varrho} H(x_0, 1+|\E\bfu|)\,\mathrm{d}x & \leq c_\eta \dashint_{B_\varrho} H_{1+|(\E\bfu)_\varrho|}(x_0,|\E\bfu - (\E\bfu)_\varrho|)\,\mathrm{d}x + \eta H(x_0, 1+|(\E\bfu)_\varrho|) \\
& \leq c H(x_0, 1+|(\E\bfu)_\varrho|) \\
& \leq c (1+|(\E\bfu)_\varrho|)^p\,.
\end{split}
\end{equation*}

With this and \eqref{eq:3.9okbis} we then get
\begin{equation*}
\begin{split}
\varrho^{\frac{\alpha-\varepsilon_0}{q}} \left(\dashint_{B_{\varrho/2}}\mu(x)(1+|\E\bfu|)^q\,\mathrm{d}x\right)^\frac{q-1}{q} & \leq \varrho^{\frac{\alpha-\varepsilon_0}{q}} \left(\dashint_{B_{\varrho}}(1+|\E\bfu|)^p\,\mathrm{d}x\right)^\frac{q-1}{q} \\
& \leq c \varrho^{\frac{\alpha-\varepsilon_0}{q}} \left(\dashint_{B_{\varrho}}(1+|\E\bfu|)^{ps_0}\,\mathrm{d}x\right)^\frac{q-p}{qs_0}\left(\dashint_{B_{\varrho}}H(x_0, 1+|\E\bfu|)^{ps_0}\,\mathrm{d}x\right)^\frac{p-1}{q} \\
& \leq c \varrho^{\frac{1}{q}(\alpha-\varepsilon_0-(q-p)n+2p\varepsilon_0)}\|1+|\E\bfu|\|_{L^{ps_0}(B_{\varrho})}(1+|(\E\bfu)_\varrho|)^\frac{p(p-1)}{q} \\
& \leq c \varrho^{\frac{\varepsilon_0}{q}}(1+|(\E\bfu)_\varrho|)^{p-1}\,.
\end{split}
\end{equation*}
A similar estimate holds for the second summand in the right hand side of \eqref{eq:stimaJ2bis}, with $\mu(x_0)$ in place of $\mu(x)$. 
Inserting these estimates in \eqref{eq:stimaJ2bis} and using the definition of $H_{1}(x_0, 1+|(\E\bfu)_\varrho|)$ we finally obtain
\begin{equation}
\begin{split}
\frac{|J_{2,1}|}{H''(x_0, 1+|(\E\bfu)_\varrho|)(1+|(\E\bfu)_\varrho|)} & \leq c \varrho^{\frac{\varepsilon_0}{q}}\,.\\
\end{split}
\label{eq:stimaJ2tris}
\end{equation}

{\bf Step 2: $(p,q)$-phase.} If \eqref{eq:4.5ok} does not hold, then \eqref{eq:stimamu1}  
is in force. Now, using \eqref{eq:smallness}, the triangle inequality and the change of shift formula, we get
\begin{equation*}
\begin{split}
\dashint_{B_{\varrho}}H(x_0,1+|\E\bfu|)\,\mathrm{d}x & \leq c \dashint_{B_{\varrho}}H(x_0,|\E\bfu-(\E\bfu)_\varrho|)\,\mathrm{d}x + c H(x_0, 1+|(\E\bfu)_\varrho|) \\
& \leq c_\eta \dashint_{B_{\varrho}}H_{1+|(\E\bfu)_\varrho|}(x_0,|\E\bfu-(\E\bfu)_\varrho|)\,\mathrm{d}x +(c+\eta)H(x_0,1+|(\E\bfu)_\varrho|) \\
& \leq c H(x_0,1+|(\E\bfu)_\varrho|)\,.
\end{split}
\end{equation*}
With this, H\"older's inequality and Jensen's inequality, we can estimate $J_{2,1}$ as
\begin{equation*}
\begin{split}
|J_{2,1}| & \leq c \dashint_{B_{\varrho/2}}\varrho^{\varepsilon_0}H'(x_0, 1+|\E\bfu|)\,\mathrm{d}x \\
& \leq c \varrho^{\varepsilon_0}\left(\dashint_{B_{\varrho}}[H'(x_0,1+|\E\bfu|)]^{\frac{2q-1}{2q-2}}\,\mathrm{d}x\right)^\frac{2q-2}{2q-1} \\
& \leq c \varrho^{\varepsilon_0}(H'(x_0,\cdot)\circ H^{-1}(x_0,\cdot))\left(\dashint_{B_{\varrho}}H(x_0,1+|\E\bfu|)\,\mathrm{d}x\right) \\
& \leq \varrho^{\varepsilon_0} H'(x_0, 1+|(\E\bfu)_\varrho|)\,,
\end{split}
\end{equation*}
which corresponds to \eqref{eq:stimaJ2tris}. 
In order to estimate $J_{2,2}$ and $J_{2,3}$, we may argue as in the proof of Lemma~\ref{lem:lemma4.1}, exploiting the smallness assumptions in \eqref{eq:5.6acming}, so we briefly sketch the proof. 

Setting $\gamma:=\left(\frac{p^*}{2}\right)'$ again, and using H\"older's inequality we have
\begin{equation*}
\begin{split}
|J_{2,2}| & \leq \left(\dashint_{B_{\varrho}} |\bfu|^{p^*}\,\mathrm{d}x\right)^\frac{1}{p^*}\left(\dashint_{B_{\varrho}} |D\bfu|^{\gamma}\,\mathrm{d}x\right)^\frac{1}{\gamma}\left(\dashint_{B_{\varrho}} |\bm\varphi|^{p^*}\,\mathrm{d}x\right)^\frac{1}{p^*} \,,
\end{split}
\end{equation*}
then the proof is similar. The only difference is the use of the Sobolev-Korn inequality \eqref{eq:2.10bogel} for $\bm\varphi$ with $|\E\bm\varphi|\leq1$.  We then obtain
\begin{equation}
|J_{2,2}| \leq c \varrho (1+|(\E\bfu)_\varrho|) \leq c \varrho H'(x_0, 1+|(\E\bfu)_\varrho|)\,.
\end{equation}
For what concerns $J_{2,3}$, using H\"older's inequality, the fact that $p^*>\frac{n}{n-1}$ and the Sobolev-Korn inequality for $\bm\varphi$ with $|\E\bm\varphi|\leq1$, we have
\begin{equation*}
\begin{split}
|J_{2,3}| & \leq c \left(\dashint_{B_\varrho} |f|^n\,\mathrm{d}x\right)^\frac{1}{n} \left(\dashint_{B_{\varrho}} |\bm\varphi|^{\frac{n}{n-1}}\,\mathrm{d}x\right)^\frac{n-1}{n} \\
& \leq c \varrho^{-\frac{1}{\beta+1}}\|f\|_{L^{n(1+\beta)}(\Omega)} \left(\dashint_{B_{\varrho}} |\bm\varphi|^{p^*}\,\mathrm{d}x\right)^\frac{1}{p^*} \\
& \leq c \varrho^{1-\frac{1}{\beta+1}}\\
& \leq c \varrho^{1-\frac{1}{\beta+1}}\,H'(x_0, 1+|(\E\bfu)_\varrho|)\,.
\end{split}
\end{equation*}
Collecting the previous estimates, we then obtain
\begin{equation}
\frac{|J_2|}{H'(x_0, 1+|(\E\bfu)_\varrho|)} \leq c \varrho^{\widetilde{\gamma}}\,,
\label{eq:stimaJ2final}
\end{equation}
where $\widetilde{\gamma}:=\min\{\varepsilon_0, 1- \frac{1}{\beta+1}\}$. The final estimate \eqref{eq:4.12ok} then follows inserting \eqref{eq:4.14ok} and \eqref{eq:stimaJ2final} into \eqref{eq:4.13ok}.
\end{proof}

\subsection{Excess decay estimates} \label{sec:excessdecay}

{In this section, we prove an excess improvement estimate for weak solutions to \eqref{eq:1.1ok}. This will be the content of  Lemma~\ref{lem:lemma4.3ok}. We start with a technical tool useful in the sequel.}
\begin{lemma}
Let $\vartheta\in(0,1)$. Assume that
 \begin{equation}
\frac{\Phi(x_0, \varrho)}{H(x_0,1+|(\E{\bf u})_{x_0, \varrho}|)}\leq \frac{\vartheta^n}{2^{q+1} c_{q}}\,,
\label{eq:smallnessbis}
\end{equation}
where $c_{q}$ is the constant of the change of shift formula \eqref{(5.4diekreu)} with $\eta=\frac{1}{2^{q+1}}$. Then,
\begin{equation}
1+|(\E\bfu)_{x_0, \varrho}| \leq 2(1+ |(\E\bfu)_{x_0, \vartheta\varrho}|)\,.
\label{eq:comparisonavebis}
\end{equation}
\label{lem:meanscomp}
\end{lemma}
\begin{proof}
As a consequence of  \eqref{(5.4diekreu)} for $\eta=\frac{1}{2^{q+1}}$ and with \eqref{eq:smallnessbis} we get
\begin{equation*}
\begin{split}
H(x_0,|(\E\bfu)_{\varrho}-(\E\bfu)_{\vartheta\varrho}|) & \leq \dashint_{B_{\vartheta\varrho}} H(x_0,|\E\bfu - (\E\bfu)_{\varrho}|)\,\mathrm{d}x \\
& \leq c_{q}\vartheta^{-n}\Phi(\varrho) + \frac{1}{2^{q+1}}H(x_0, 1+|(\E\bfu)_{\varrho}|) \\
& \leq \frac{1}{2^{q}} H(x_0, 1+|(\E\bfu)_{\varrho}|)\,,
\end{split}
\label{eq:computation}
\end{equation*}
whence, passing to $[H(x_0,\cdot)]^{-1}$ and taking into account \eqref{(2.3a)}, we obtain
\begin{equation*}
|(\E\bfu)_{\varrho}-(\E\bfu)_{\vartheta\varrho}| \leq \frac{1}{2}(1+|(\E\bfu)_{\varrho}|)\,.
\end{equation*}
Now,
\begin{equation*}
\begin{split}
1+|(\E\bfu)_{\varrho}| \leq |(\E\bfu)_{\varrho}-(\E\bfu)_{\theta\varrho}|+1+|(\E\bfu)_{\theta\varrho}| \leq \frac{1}{2}(1+|(\E\bfu)_{\varrho}|) + 1+ |(\E\bfu)_{\theta\varrho}|\,,
\end{split}
\end{equation*}
whence \eqref{eq:comparisonavebis} follows by re-absorbing the first term of the right-hand side into the left. 
\end{proof}

{We are now in position to prove the excess improvement estimate.}

\begin{lemma}\label{lem:lemma4.3ok}
For any fixed $\theta\in(0,\frac{1}{8})$, there exists $\varepsilon_1=\varepsilon_1(n,\nu,L,p,q,[\mu]_{C^\alpha},\theta)\in(0,1)$ such that if
\begin{equation}
\frac{\Phi(x_0,\varrho)}{H(x_0, 1+|(\E\bfu)_{x_0,\varrho}|)}+\varrho^{\widetilde{\gamma}}\leq \varepsilon_1\,,\quad \varrho\leq\theta^n\,,
\label{eq:4.18ok}
\end{equation} 
then
\begin{equation}
\Phi(x_0,\theta\varrho) \leq c_{\rm dec} \theta^2\left[\Phi(x_0,\varrho) + \varrho^{\widetilde{\gamma}}H(x_0, 1+|(\E\bfu)_{x_0,\varrho}|)\right]
\label{eq:4.19ok}
\end{equation}
for some constant $c_{\rm dec} = c_{\rm dec}(n,\nu,L,p,q,[\mu]_{C^\alpha},\|\E\bfu\|_{L^p(\Omega)})\geq1$.
\end{lemma}

\begin{proof}
{\bf Step~1:} By virtue of \eqref{eq:caccioppoliII} we have
\begin{equation}
\begin{split}
& \dashint_{B_{\theta\varrho}} H_{1+|(\E\bfu)_{\theta\varrho}|}(x_0, |\E\bfu-(\E\bfu)_{\theta\varrho}|)\,\mathrm{d}x \\
& \,\,\,\, \leq c \dashint_{B_{\theta\varrho}} H_{1+|(\E\bfu)_{\theta\varrho}|}\left(x_0,\frac{|{\bf u}-\bm\ell_{\theta\varrho}|}{\theta\varrho}\right)\,\mathrm{d}x + c  (\theta\varrho)^{\widetilde{\gamma}}\,H(x_0,1+|(\E\bfu)_{\theta\varrho}|)\,.
\end{split}
\label{eq:caccioppolitheta}
\end{equation} 
Choosing
\begin{equation}
\sqrt{\varepsilon_1}\leq \frac{\theta^n}{8}\,,
\label{eq:4.20ok}
\end{equation}
we obtain
\begin{equation}
|(\E\bfu)_{\theta\varrho}| \leq \frac{\sqrt{\varepsilon_1}}{8} |(\E\bfu)_{\varrho}| <  |(\E\bfu)_{\varrho}| \,.
\label{eq:comparisonave}
\end{equation}
Therefore, combining \eqref{eq:4.18ok} and \eqref{eq:4.20ok} 
 we finally obtain
\begin{equation}
\begin{split}
(\theta\varrho)^{\widetilde{\gamma}} \leq \left[\frac{\Phi(\varrho) }{H(x_0,1+|(\E\bfu)_{\varrho}|)} + \varrho^{\widetilde{\gamma}}\right]^2 \leq c \theta^2 \left[\frac{\Phi(\varrho) }{H(x_0,1+|(\E\bfu)_{\varrho}|)} + \varrho^{\widetilde{\gamma}}\right]\,.
\end{split}
\end{equation}
Inserting this estimate in \eqref{eq:caccioppolitheta} and using \eqref{eq:comparisonave}, we then obtain
\begin{equation}
\begin{split}
\Phi(\theta\varrho) \leq c \dashint_{B_{\theta\varrho}} H_{1+|(\E\bfu)_{\theta\varrho}|}\left(x_0,\frac{|{\bf u}-\bm\ell_{\theta\varrho}|}{\theta\varrho}\right)\,\mathrm{d}x + c \theta^2 \left[\Phi(x_0,\varrho) + \varrho^{\widetilde{\gamma}}H(x_0, 1+|(\E\bfu)_{x_0,\varrho}|)\right]\,.
\end{split}
\label{eq:caccioppolithetabis}
\end{equation} 
\noindent
{\bf Step~2:} Following \cite[Lemma~4.2]{CeladaOk}, we prove that Theorem~\ref{AStokes} and the subsequent remark can be applied to function
\begin{equation}
{\bf v}:= \frac{{\bfu - \bm\ell_{\varrho}}}{1+|(\E\bfu)_{\varrho}|}\,.
\end{equation}

Setting
\begin{equation}
G(t):= \frac{H_{1+|(\E\bfu)_{\varrho}|}(x_0, (1+|(\E\bfu)_{\varrho}|)t)}{H(x_0, 1+|(\E\bfu)_{\varrho}|)}\,,
\end{equation}
by using the fact that $G$ is indeed a shifted $N$-function it can be seen that
\begin{equation}
\tilde{c}G(t)\geq t^2\,, \quad t\in[0,1]\,,
\label{eq:lowboundG}
\end{equation}
for some constant $\tilde{c}>0$. Moreover, by Corollary~\ref{corollary3.2} we have that 
\begin{equation}
\begin{split}
\left(\dashint_{B_{\varrho}} \left[G\left(\frac{|\E(\bfu-\bm\ell_{\varrho})|}{1+|(\E\bfu)_{\varrho}|}\right)\right]^\sigma\,\mathrm{d}x \right)^{\frac{1}{\sigma}} & =\left(\dashint_{B_{\varrho}} \left[\frac{H_{1+|(\E\bfu)_{\varrho}|}(x_0, |\E\bfu-(\E\bfu)_{\varrho}|)}{H(x_0, 1+|(\E\bfu)_{\varrho}|)}\right]^\sigma\,\mathrm{d}x \right)^{\frac{1}{\sigma}} \\
&\leq c_{\rm high} \frac{1}{H(x_0, 1+|(\E\bfu)_{\varrho}|)}\dashint_{B_{\varrho}} H_{1+|(\E\bfu)_{\varrho}|}(x_0, |\E\bfu-(\E\bfu)_{\varrho}|)\,\mathrm{d}x \\
&\,\,\,\,\,\,\,\,\,\, + c_{\rm high}\varrho^{\widetilde{\gamma}} \\ 
& \leq c_{\rm high} \left [\frac{\Phi(\varrho)}{H(x_0, 1+|(\E\bfu)_{\varrho}|)}+\varrho^{\widetilde{\gamma}}\right]\,.
\end{split}
\label{eq:highintestimate}
\end{equation} 
Now, setting
\begin{equation}
\varepsilon:=\max\{c_{\rm Stokes}, \sqrt{c_{\rm high} \tilde{c}}\}\left [\frac{\Phi(\varrho)}{H(x_0, 1+|(\E\bfu)_{\varrho}|)}+\varrho^{\widetilde{\gamma}}\right]^{\frac{1}{2}} \leq \max\{c_{\rm Stokes}, \sqrt{c_{\rm high} \tilde{c}}\} \sqrt{\varepsilon_1}\,,
\label{eq:choiceeps}
\end{equation}
we can choose $\varepsilon_1$ such that $\varepsilon<1$. Moreover, combining with \eqref{eq:highintestimate} and \eqref{eq:lowboundG}, we get
\begin{equation}
\begin{split}
\left(\dashint_{B_{\varrho}} \left[G\left(\frac{|\E(\bfu-\bm\ell_{\varrho})|}{1+|(\E\bfu)_{\varrho}|}\right)\right]^\sigma\,\mathrm{d}x \right)^{\frac{1}{\sigma}} \leq c_{\rm high} \left [\frac{\Phi(\varrho)}{H(x_0, 1+|(\E\bfu)_{\varrho}|)}+\varrho^{\widetilde{\gamma}}\right] & = c_{\rm high}\frac{\varepsilon^2}{\max\{c_{\rm Stokes}, \sqrt{c_{\rm high} \tilde{c}}\}^2} \\
&\leq G(\varepsilon) \,.
\end{split}
\label{eq:highintestimatebis}
\end{equation} 
Inserting \eqref{eq:choiceeps} and \eqref{eq:4.18ok} into \eqref{eq:4.12ok}, written for $\bfv$ and for some $\bm\varphi\in C_{0,{\rm div}}^\infty(B_{\varrho})$ with $\|D\bm\varphi\|_{L^\infty(B_{\varrho})}\leq1$, we obtain
\begin{equation}
\begin{split}
\dashint_{B_{\varrho}}\mathcal{A}\left(\frac{\E(\bfu-\bm\ell_{\varrho})}{1+|(\E\bfu)_{\varrho}|}\right):\E\bm\varphi\,\mathrm{d}x \leq c_{\rm Stokes} \frac{[\omega(\varepsilon_1^\frac{1}{2})+\varepsilon_1]^{\frac{1}{2}}}{\max\{c_{\rm Stokes}, \sqrt{c_{\rm high} \tilde{c}}\}}\varepsilon\,.  
\end{split}
\end{equation}
Then, choosing $\varepsilon_1$ small enough, the assumptions of Theorem~\ref{AStokes} and Remark~\ref{rem:remarkastokes} for function $\bfv$ are in force. We denote by ${\bf h}$ the $\mathcal{A}$-Stokes function in $B_\varrho$ such that ${\bf h}=\bfv$ on $\partial B_\varrho$. Assertion \eqref{eq:lemma2.7bis} gives
\begin{equation}
\begin{split}
\frac{1}{H(x_0, 1+|(\E\bfu)_{\varrho}|)}\dashint_{B_{\varrho}} H_{1+|(\E\bfu)_{\varrho}|}(x_0, |\E\bfu-(\E\bfu)_{\varrho}- (1+|(\E\bfu)_{\varrho}|)\E{\bf h}|)\,\mathrm{d}x & = \dashint_{B_{\varrho}} G(|\E\bfv-\E{\bf h}|)\,\mathrm{d}x \\
& \leq \kappa G(\varepsilon)\,.
\end{split}
\label{eq:4.15celok}
\end{equation}
Moreover, 
we also have
\begin{equation*}
G(\varepsilon) \leq c \frac{H(x_0,(1+|(\E\bfu)_{\varrho}|)(1+\varepsilon))}{H(x_0,1+|(\E\bfu)_{\varrho}|)(1+\varepsilon)^2}\varepsilon^2 \leq c \varepsilon^2\,,
\end{equation*}
whence
\begin{equation}
\dashint_{B_{\varrho}} H_{1+|(\E\bfu)_{\varrho}|}(x_0, |\E\bfu-(\E\bfu)_{\varrho}- (1+|(\E\bfu)_{\varrho}|)\E{\bf h}|)\,\mathrm{d}x \leq \bar{c}\kappa \left[ \Phi(\varrho) + \varrho^{\widetilde{\gamma}} H(x_0, 1+|(\E\bfu)_{\varrho}|) \right]
\label{eq:4.16celok}
\end{equation}
for some constant $\bar{c}>0$. \\
\noindent
{\bf Step~3:} {Now, we go back to the estimate of the right hand side of \eqref{eq:caccioppolithetabis}. As a consequence of Lemma~\ref{lem:thm3.2ok}, with \eqref{eq:comparisonave} and \eqref{eq:comparisonavebis}, we get
\begin{equation*}
\begin{split}
&\dashint_{B_{\theta\varrho}}   H_{1+|(\E\bfu)_{\theta\varrho}|}\left(x_0,\frac{|{\bf u}-\bm\ell_{\theta\varrho}|}{\varrho}\right)\,\mathrm{d}x \\
& \leq c \left(1+ [\mu]_{C^\alpha}\|\E\bfu\|^{q-p}_{L^p(\Omega)}\theta^{\alpha-\frac{(q-p)n}{p}}\right)\dashint_{B_{\theta\varrho}} \left[H_{1+|(\E\bfu)_{\theta\varrho}|}(x_0, |\E\bfu-(\E\bfu)_{\theta\varrho}|)\right]\,\mathrm{d}x \\
& \leq c \left(1+ [\mu]_{C^\alpha}\|\E\bfu\|^{q-p}_{L^p(\Omega)}\theta^{\alpha-\frac{(q-p)n}{p}}\right)\dashint_{B_{\theta\varrho}} \left[H_{1+|(\E\bfu)_{\varrho}|}(x_0, |\E\bfu-(\E\bfu)_{\theta\varrho}|)\right]\,\mathrm{d}x\,.
\end{split}
\end{equation*}
Now, the integral in the right hand side above can be treated exactly as in \cite[pp. 24--25]{CeladaOk} starting from \eqref{eq:4.16celok}, to obtain the estimate}
\begin{equation}
\dashint_{B_{\theta\varrho}} H_{1+|(\E\bfu)_{\varrho}|}(x_0, |\E\bfu-(\E\bfu)_{\theta\varrho}|)\,\mathrm{d}x \leq c \theta^2 \left[\frac{\kappa}{\theta^{n+2}}+1\right] \left [\frac{\Phi(\varrho)}{H(x_0, 1+|(\E\bfu)_{\varrho}|)}+\varrho^{\widetilde{\gamma}}\right]\,.
\end{equation}
Thus, we omit further details. This concludes the proof of \eqref{eq:4.19ok}.
\end{proof}

{Now, we prove that the previous excess improvement estimate can be iterated at each scale.} For this, we introduce the Morrey-type excess
\begin{equation}
\Theta(x_0,\varrho):=\varrho^\frac{1}{2} [H(x_0,\cdot)]^{-1}\left(\dashint_{B_{\varrho}}H(x_0,1+|D \bfu|)\,\mathrm{d}x\right)\,.
\label{eq:remainder2}
\end{equation}
\begin{lemma}
Let $\Phi(x_0,\varrho)$ and  $\Theta(x_0,\varrho)$ 
be defined as in \eqref{eq:excess1} and \eqref{eq:remainder2}, 
respectively. 
Then there exist constants $\delta_*$, $\varepsilon_*, \varrho_*\in(0,1]$, $M\geq1$ and $\vartheta$ such that the following holds:
if the conditions 
\begin{equation}
\frac{\Phi(x_0,\varrho)}{H(x_0,1+|(\E\bfu)_{x_0,\varrho}|)}\leq \varepsilon_*\,,\qquad \Theta(x_0,\varrho)\leq\delta_*\,, \qquad |(\bfu)_{x_0,\varrho}|\leq\frac{1}{2}M
\label{eq:0step}
\end{equation}
hold on $B_\varrho(x_0)\subseteq\Omega$ for $\varrho\in(0,\varrho_*]$, then 
\begin{equation}
\frac{\Phi(x_0,\vartheta^m\varrho)}{H(x_0, 1+|(\E\bfu)_{x_0,\vartheta^m\varrho}|)}\leq \varepsilon_*\,,\qquad  \Theta(x_0,\vartheta^m\varrho)\leq\delta_*\,, \qquad |(\bfu)_{x_0,\vartheta^m\varrho}|\leq M
\label{eq:kstep}
\end{equation}
for every $m=0,1,\dots.$ In particular, this would imply
\begin{equation}
(\vartheta^m\varrho)^\frac{p}{2}(|D\bfu|^p)_{x_0,\vartheta^m\varrho} \leq \delta_*^p\,, \quad m=0,1,\dots\,.
\label{eq:equivalentp}
\end{equation}
Moreover, for any $\beta\in(0,1)$ the following Morrey-type estimate holds:
\begin{equation}
\Theta(y,r)\leq c\delta_*\left(\frac{r}{\varrho}\right)^\frac{\beta}{2}
\label{(5.10Stroffo)}
\end{equation}
for all $y\in B_{\varrho/2}(x_0)$ and $r\in(0,\varrho/2]$.
\label{lem:lemma3.13}
\end{lemma}

\proof
As usual, we omit the explicit dependence on $x_0$. Let $\vartheta\in(0,1)$ be such that
\begin{equation}
\vartheta\leq \min\left\{(8c_{\rm dec}2^{q-1})^{-\frac{1}{2}},\frac{1}{2^{2q}},\frac{1}{2^{\frac{q}{p (1-\beta)}}}\right\}\,,
\label{eq:choosetheta}
\end{equation}
where $c_{\rm dec}$ is the constant of Lemma~\ref{lem:lemma4.3ok} depending only on $n,\nu,L,p,q,[\mu]_{C^\alpha},\|\E\bfu\|_{L^p(\Omega)}$.
Correspondingly, let  $\varepsilon_1=\varepsilon_1(n,\nu,L,p,q,[\mu]_{C^\alpha},\|\E\bfu\|_{L^p(\Omega)}, \vartheta)$ be the constant of Lemma~\ref{lem:lemma4.3ok}, applied with the choice $\varepsilon=\vartheta^{n+2}$. We choose $\varepsilon_*>0$ such that
\begin{equation}
\varepsilon_*\leq\min\left\{\frac{\varepsilon_1}{2},\frac{\vartheta^n}{\max\{2c_{\frac{1}{2}},2^{q+1} c_{q}\}}\right\}\,,
\label{(3.46verena)}
\end{equation}
where $c_{\frac{1}{2}}$ is the constant in the change-shift formula \eqref{(5.4diekreu)} with $\eta=\frac{1}{2}$, while $c_q$ is obtained with $\eta=\frac{1}{2^{q+1}}$. 
Moreover, we choose a radius $0 < \varrho_*\leq1$ such that
\begin{equation}
\varrho_*<\min \left\{\varepsilon_*^\frac{1}{\widetilde{\gamma}}, \left(\frac{\vartheta^n (1-\vartheta^{p^*-\frac{p}{2}})}{2\delta_*^p}\right)^\frac{1}{p^*-\frac{p}{2}}\right\}\,.
\label{(3.48verena)}
\end{equation}
As a consequence, $\varepsilon_*$ and $\varrho_*$ have the same dependencies as $\varepsilon_1$. 
{We argue by induction on $m$. Since \eqref{eq:kstep} are trivially true for $m=0$ by assumption \eqref{eq:0step}, our aim is to show that if \eqref{eq:kstep} holds for some $m\geq1$, then the corresponding inequalities hold with $m+1$ in place of $m$.
Setting
\begin{equation*}
E(B_{\vartheta^m\varrho}):=\dashint_{B_{\vartheta^m\varrho}}H(x_0, 1+|D \bfu|)\,\mathrm{d}x\,,
\end{equation*}
in order to prove the second inequalities in \eqref{eq:kstep} it will suffice to show that
\begin{equation}
E(B_{\vartheta^m\varrho})\leq H\left(x_0, \frac{\delta_*}{(\vartheta^m\varrho)^\frac{1}{2}}\right)\,.
\label{eq:equivalent}
\end{equation}
With \eqref{eq:kstep} at step $m$, the shift-change formula \eqref{(5.4diekreu)} with $\eta=\frac{1}{2}$, \eqref{(3.46verena)} and Lemma~\ref{lem:kornorlicz}, we have the estimate
\begin{equation}
\begin{split}
E(B_{\vartheta^{m+1}\varrho}) & \leq 2^{q-1}\left(\dashint_{B_{\vartheta^{m+1}\varrho}}H(x_0,|D \bfu - (D \bfu)_{\vartheta^{m}\varrho}|)\,\mathrm{d}x + H(x_0,1+|(D \bfu)_{\vartheta^{m}\varrho}|)\right) \\
& \leq 2^{q-1}\left(c_{\frac{1}{2}}\dashint_{B_{\vartheta^{m+1}\varrho}}H_{1+|(\E\bfu)_{\vartheta^{m}\varrho}|}(x_0,|D \bfu - (D \bfu)_{\vartheta^{m}\varrho}|)\,\mathrm{d}x \right. \\
& \,\,\,\,\,\,\,\,\,\,\,\,\,\,\,\,\,\, \left.\vphantom{c_{\frac{1}{2}}\dashint_{B_{\vartheta^{m+1}\varrho}}H_{1+|(\E\bfu)_{\vartheta^{m}\varrho}|}(x_0,|D \bfu - (D \bfu)_{\vartheta^{m}\varrho}|)\,\mathrm{d}x}+\frac{1}{2}H(x_0,1+|(\E \bfu)_{\vartheta^{m}\varrho}|)+ H(x_0,1+|(D \bfu)_{\vartheta^{m}\varrho}|)\right) \\
& \leq 2^{q-1}\left(c_{\frac{1}{2}}\theta^{-n}\dashint_{B_{\vartheta^{m}\varrho}}H_{1+|(\E\bfu)_{\vartheta^{m}\varrho}|}(x_0,|D \bfu - (D \bfu)_{\vartheta^{m}\varrho}|)\,\mathrm{d}x +\frac{3}{2}E(B_{\vartheta^{m}\varrho}) \right) \\
&  \leq 2^{q-1}\left(c_{\frac{1}{2}}\theta^{-n}\Phi(\vartheta^m\varrho) +\frac{3}{2}E(B_{\vartheta^{m}\varrho}) \right)\\
& \leq 2^{q-1}\left(c_{\frac{1}{2}}\vartheta^{-n}\varepsilon_*+\frac{3}{2}\right) E(B_{\vartheta^m\varrho})\\
&\leq 2^{q-1}\left(c_{\frac{1}{2}}\vartheta^{-n}\varepsilon_*+\frac{3}{2}\right)\vartheta^\frac{1}{2} H\left(x_0,\frac{\delta_*}{(\vartheta^{m+1}\varrho)^\frac{1}{2}}\right)\\
& \leq  H\left(x_0, \frac{\delta_*}{(\vartheta^{m+1}\varrho)^\frac{1}{2}}\right)\,,
\end{split}
\label{estim1}
\end{equation}}
which proves $\eqref{eq:kstep}_2$ for $m+1$. With this at hand, since the function $G(t^\frac{1}{p}):=H(x_0,t^\frac{1}{p})$ is convex, from Jensen's inequality we get
\begin{equation*}
G\left(\left(\dashint_{B_{\vartheta^i\varrho}}|D\bfu|^p \, \mathrm{d}x\right)^\frac{1}{p}\right) \leq \dashint_{B_{\vartheta^mi\varrho}}G(1+|D\bfu|)\, \mathrm{d}x \leq G\left(\frac{\delta_*}{(\vartheta^i\varrho)^\frac{1}{2}}\right)\,,
\end{equation*}
whence, passing to $G^{-1}$, we obtain \eqref{eq:equivalentp}.
Now, we prove by induction the first inequality in \eqref{eq:kstep} for $m+1$.
From \eqref{eq:kstep} at step $k$ and the choice of $\varrho_*$ as in \eqref{(3.48verena)}, we have
\begin{equation*}
\begin{split}
\frac{\Phi(\vartheta^m\varrho)}{H(x_0, 1+|(\E\bfu)_{\vartheta^m\varrho}|)} & \leq \varepsilon_*<2\varepsilon_*\leq\varepsilon_1\,,\\
(\vartheta^m\varrho)^{\widetilde{\gamma}} & <2\varepsilon_*\leq \varepsilon_1\,,
\end{split}
\end{equation*}
and
\begin{equation*}
\begin{split}
\Phi(\vartheta^m\varrho) + (\vartheta^m\varrho)^{\widetilde{\gamma}} H(x_0, 1+|(\E\bfu)_{\vartheta^m\varrho}|) & \leq 4 \epsilon_* H(x_0, 1+|(\E\bfu)_{\vartheta^m\varrho}|)\,.
\end{split}
\end{equation*}
Then, by virtue of Lemma~\ref{lem:lemma4.3ok} and Lemma~\ref{lem:meanscomp} applied with radius $\vartheta^m\varrho$ in place of $\varrho$, and recalling the choice of $\vartheta$ \eqref{eq:choosetheta}, we get
\begin{equation*}
\begin{split}
\Phi(\vartheta^{m+1}\varrho) &\leq 2c_{\rm dec}\vartheta^2[\Phi(\vartheta^m\varrho) + (\vartheta^m\varrho)^{\widetilde{\gamma}} H(x_0, 1+|(\E\bfu)_{\vartheta^m\varrho}|)] \\
&\leq 8c_{\rm dec}\epsilon_*\vartheta^2H(x_0, 1+|(\E\bfu)_{\vartheta^m\varrho}|)\\
& \leq \epsilon_* H(x_0, 1+|(\E\bfu)_{\vartheta^{m+1}\varrho}|)\,.
\end{split}
\end{equation*}
To conclude the proof, we are left to prove $\eqref{eq:kstep}_3$ and \eqref{(5.10Stroffo)}. 
We use $\eqref{eq:0step}_3$, the Poincar\'e inequality with $\eqref{eq:kstep}_2$, $\varrho\leq \varrho_*$ and \eqref{(3.48verena)}, \eqref{eq:equivalentp} for every $i=0,\dots,m$ and we get
\begin{equation}
\begin{split}
|(\bfu)_{\vartheta^{m+1}\varrho}| & \leq |(\bfu)_{\varrho}| + \vartheta^{-n} \sum_{i=0}^m\dashint_{B_{\vartheta^i\varrho}} |\bfu-(\bfu)_{\vartheta^i\varrho}|^{p^*} \, \mathrm{d}x \\
& \leq \frac{1}{2} M + \vartheta^{-n} \sum_{i=0}^m (\vartheta^i\varrho)^{p^*} \dashint_{B_{\vartheta^i\varrho}}|D\bfu|^p \, \mathrm{d}x \\
& \leq \frac{1}{2} M + \delta_*^p\vartheta^{-n} \sum_{i=0}^m (\vartheta^i\varrho)^{p^*-\frac{p}{2}} \\
& = \frac{1}{2} M + \delta_*^{p}\vartheta^{-n} \frac{\varrho^{p^*-\frac{p}{2}}}{1-\vartheta^{p^*-\frac{p}{2}}} \\
& \leq \frac{1}{2} M + \frac{1}{2}\leq M\,.
\end{split}
\end{equation}
Finally, setting $G(t):=H(x_0,t)$, since the iteration starting from $m=0$ of the estimate $G^{-1}(E(B_{\vartheta^{m+1}\varrho}))\leq 2^{\frac{q}{p}}G^{-1}(E(B_{\vartheta^{m}\varrho}))$, obtained by \eqref{estim1} and \eqref{(2.3a)}, with \eqref{eq:choosetheta} yields 
\begin{equation*}
(\vartheta^m\varrho)^\frac{1-\beta}{2}G^{-1}(E(B_{\vartheta^m\varrho}))\leq \varrho^\frac{1-\beta}{2}G^{-1}(E(B_{\varrho}))\leq \delta_*\varrho^{-\frac{\beta}{2}}\,,
\end{equation*}
and this estimate \emph{a fortiori} holds if we consider $E(B_{\vartheta^m\varrho}(y))$ for $y\in B_{\varrho/2}$ in place of $E(B_{\vartheta^m\varrho})$, we deduce the Morrey-type estimate
\begin{equation*}
r^\frac{1-\beta}{2}G^{-1}(E(B_{r}(y)))\leq c\delta_*\varrho^{-\frac{\beta}{2}}
\end{equation*}
for all $y\in B_{\varrho/2}$ and $r\leq\varrho/2$, which is equivalent to \eqref{(5.10Stroffo)}.
This concludes the proof.
\endproof

\section{Proof of Theorem~\ref{thm:thm1.1ok}} \label{sec:proofmainthm}

We are now in position to prove Theorem~\ref{thm:thm1.1ok}.

\begin{proof}
Let $\delta_*$, $\varepsilon_*, \varrho_*\in(0,1]$ be the constants of Lemma~\ref{lem:lemma3.13}. We define
\begin{equation*}
\Omega_0:=\left\{z_0\in\Omega:\,\, \bfu\in C^\beta(U_{z_0};\R^n) \mbox{ for every $\beta\in(0,1)$ and for some }U_{z_0}\subset\Omega\right\}\,,
\end{equation*}
where $U_{z_0}$ is an open neighborhood of $z_0$. Assuming that $x_0\in\Omega$ complies with

\begin{equation}
\begin{split}
&\lim_{\varrho\searrow0} \dashint_{B_{\varrho}(x_0)}|D\bfu-(D\bfu)_{x_0,\varrho}|\,\mathrm{d}x = 0\,, \\
M_{x_0}:=&\mathop{\lim\sup}_{\varrho\searrow 0}\left[\dashint_{B_\varrho(x_0)}H(x,1+|\E\bfu|)\,\mathrm{d}x + (|D\bfu|^p)_{x_0,\varrho} + |(\bfu)_{x_0,\varrho}|\right]<+\infty\,;
\end{split}
\label{eq:5.24ok}
\end{equation}
i.e., $x_0\in\Omega\setminus(\Sigma_1\cup\Sigma_2)$, we will prove that $x_0\in\Omega_0$. We fix $\beta\in(0,1)$ and choose $t\in(0,1)$ such that
\begin{equation}
\frac{1}{q} = t + \frac{1-t}{qs_0}\,,
\label{eq:ok(5.25)}
\end{equation}
where $s_0>1$ is the exponent of Lemma~\ref{lem:higint}.
We also set
\begin{equation}
\varepsilon\leq\min \left\{\left[\frac{1}{c_*}\left(\frac{\varepsilon_*}{2}\right)^\frac{p}{2}\frac{1}{\bar{c}(M_{x_0}+1)^{\frac{q(1-t)}{p}}}\right]^\frac{1}{pt}, 1 \right\}\,,
\label{eq:(5.26)ok}
\end{equation}
where $c_*$ is the constant $c_\eta$ in the change of shift formula for $\eta_*:=\frac{\varepsilon_*}{2^{n+1} (M_{x_0}+1) c_H}$. By virtue of \eqref{eq:5.24ok}, we can find $0<\varrho$ with 
\begin{equation}
\varrho \leq  \min\left\{\frac{\delta_*}{[H(x_0,\cdot)]^{-1}(\tilde{c}^\frac{1}{s_0}(M_{x_0}+1)^{q}) }, \varrho_*\right\}\,.
\label{eq:smallvarrho}
\end{equation}
and $B_{3\varrho}(x_0)\subset\subset\Omega$ such that
\begin{equation}
\begin{split}
&\dashint_{B_{\varrho}(x_0)}|D\bfu-(D\bfu)_{x_0,\varrho}|\,\mathrm{d}x <\varepsilon\,, \\
& \dashint_{B_{2\varrho}(x_0)}H(x,1+|\E\bfu|)\,\mathrm{d}x + (|D\bfu|^p)_{x_0,\varrho} + |(\bfu)_{x_0,\varrho}|<M_{x_0}+1\,.
\end{split}
\label{eq:4.40ok}
\end{equation}
%
\\
{\bf Step~1:} {We first establish the higher integrability of $H(x_0, 1+|D\bfu|)$ on $B_\varrho(x_0)$. Namely, we show that
\begin{equation}
\dashint_{B_{\varrho}(x_0)}[H(x_0,1+|D\bfu|)]^{s_0}\,\mathrm{d}x \leq \tilde{c} (M_{x_0}+1)^{qs_0} \,,
\label{eq:highintgradients}
\end{equation}
where $s_0>1$ is the exponent of Lemma~\ref{lem:higint}.

If $\mu$ satisfies \eqref{eq:3.8ok}, since \eqref{eq:4.40ok} implies \eqref{eq:smallnesshighint}, from \eqref{eq:3.9okbis} we infer 
\begin{equation*}
\begin{split}
\left(\dashint_{B_{\varrho}(x_0)}[H(x,1+|\E\bfu|)]^{s_0}\,\mathrm{d}x\right)^{\frac{1}{s_0}} & \leq c \dashint_{B_{2\varrho}(x_0)}H(x,1+|\E\bfu|)\,\mathrm{d}x \\ 
  & \leq c (M_{x_0}+1)\,,
\end{split}
\end{equation*}
whence
\begin{equation}
\dashint_{B_{\varrho}(x_0)}[H(x,1+|\E\bfu|)]^{s_0}\,\mathrm{d}x  \leq  c(M_{x_0}+1)^{s_0}\,.
\label{eq:highintsymm}
\end{equation}
Now, from Korn's inequality, Poincar\'e inequality, \eqref{eq:4.40ok}, \eqref{eq:3.9okbis} and $\varrho\leq1$ we have
\begin{equation*}
\begin{split}
\dashint_{B_\varrho(x_0)} (1+|D\bfu|)^{ps_0}\,\mathrm{d}x & \leq c \dashint_{B_\varrho(x_0)} (1+|\E\bfu|)^{ps_0}\,\mathrm{d}x + c \left(\dashint_{B_\varrho(x_0)} \left |\frac{\bfu - (\bfu)_{x_0,\varrho}}{\varrho}\right|\,\mathrm{d}x\right)^{ps_0} \\
& \leq c \dashint_{B_\varrho(x_0)} (1+|\E\bfu|)^{ps_0}\,\mathrm{d}x + c \left(\dashint_{B_\varrho(x_0)} |D\bfu|\,\mathrm{d}x\right)^{ps_0} \\
& \leq c \left(\dashint_{B_{2\varrho}(x_0)} (1+|\E\bfu|)^{p}\,\mathrm{d}x \right)^{s_0} + c(M_{x_0}+1)^{s_0}\,.
\end{split}
\end{equation*}
Analogously, using also H\"older's inequality and $\varrho\leq1$, we have
\begin{equation*}
\begin{split}
\dashint_{B_\varrho(x_0)} [\mu(x_0)(1+|D\bfu|)^q]^{s_0}\,\mathrm{d}x  & \leq c\varrho^{\alpha s_0}\left [\dashint_{B_\varrho(x_0)} (1+|\E\bfu|)^{qs_0}\,\mathrm{d}x + \left(\dashint_{B_\varrho(x_0)} |D\bfu|\,\mathrm{d}x\right)^{qs_0} \right] \\
& \leq c\varrho^{\alpha s_0}\left [\left(\dashint_{B_{2\varrho}(x_0)} (1+|\E\bfu|)^{p}\,\mathrm{d}x\right)^\frac{qs_0}{p} + \left(\dashint_{B_\varrho(x_0)} |D\bfu|\,\mathrm{d}x\right)^{qs_0} \right] \\
& \leq c\varrho^{(\alpha-\frac{(q-p)}{p}) s_0}\|1+|\E\bfu|\|_{L^p}^{q-p}\left(\dashint_{B_{2\varrho}(x_0)} (1+|\E\bfu|)^{p}\,\mathrm{d}x\right)^{s_0} \\
& \,\,\,\,\,\, + c\varrho^{\alpha s_0} \left(\dashint_{B_\varrho(x_0)} |D\bfu|\,\mathrm{d}x\right)^{qs_0}  \\
& \leq c \left(\dashint_{B_{2\varrho}(x_0)} (1+|\E\bfu|)^{p}\,\mathrm{d}x\right)^{s_0} + c (M_{x_0}+1)^\frac{qs_0}{p} \,.
\end{split}
\end{equation*}
Combining the previous estimates, we then obtain
\begin{equation}
\begin{split}
\dashint_{B_{\varrho}(x_0)}[H(x_0,1+|D\bfu|)]^{s_0}\,\mathrm{d}x  & \leq c \left(\dashint_{B_{2\varrho}(x_0)}H(x,1+|\E\bfu|)\,\mathrm{d}x\right)^{s_0} \\
& \leq c_1 (M_{x_0}+1)^\frac{qs_0}{p} \,.
\end{split}
\label{eq:highintgrad}
\end{equation} }
{If \eqref{eq:3.8ok} doesn't hold true, then it can be shown that $H(x_0,t)\leq c H(x,t)$. In this case, with \eqref{eq:4.40ok} and  \eqref{eq:3.6ok}, \eqref{eq:stimanormu}, we get
\begin{equation}
\left(\dashint_{B_{r}(x_0)}[H(x,1+|\E\bfu|)]^{s_0}\,\mathrm{d}x\right)^{\frac{1}{s_0}} \leq c \dashint_{B_{2r}(x_0)}H(x,1+|\E\bfu|)\,\mathrm{d}x\,,
\label{eq:highintlarge}
\end{equation}
whence assertion \eqref{eq:highintsymm} can still be inferred. Now, using Korn's inequality, Poincar\'e inequality, \eqref{eq:4.40ok}, \eqref{eq:highintlarge} we have
\begin{equation*}
\begin{split}
\dashint_{B_\varrho(x_0)} (1+|D\bfu|)^{ps_0}\,\mathrm{d}x & \leq c \dashint_{B_\varrho(x_0)} (1+|\E\bfu|)^{ps_0}\,\mathrm{d}x + c \left(\dashint_{B_\varrho(x_0)} |D\bfu|\,\mathrm{d}x\right)^{ps_0} \\
& \leq c \left(\dashint_{B_{2\varrho}(x_0)} H(x,1+|\E\bfu|)\,\mathrm{d}x \right)^{s_0} + c(M_{x_0}+1)^{s_0}\,.
\end{split}
\end{equation*}
In a similar way, using also that $\mu\leq1$, we have
\begin{equation*}
\begin{split}
\dashint_{B_\varrho(x_0)} [\mu(x_0)(1+|D\bfu|)^q]^{s_0}\,\mathrm{d}x  & \leq c\left [\dashint_{B_\varrho(x_0)} [\mu(x_0)(1+|\E\bfu|)^q]^{s_0}\,\mathrm{d}x + \left(\dashint_{B_\varrho(x_0)} |D\bfu|\,\mathrm{d}x\right)^{qs_0} \right] \\
& \leq c \left(\dashint_{B_{2\varrho}(x_0)} H(x_0, 1+|\E\bfu|)\,\mathrm{d}x\right)^{s_0} + c (M_{x_0}+1)^{qs_0}\\
& \leq c \left(\dashint_{B_{2\varrho}(x_0)} H(x, 1+|\E\bfu|)\,\mathrm{d}x\right)^{s_0} + c (M_{x_0}+1)^{qs_0}\,,
\end{split}
\end{equation*}
which combined with the previous one gives
\begin{equation}
\begin{split}
\dashint_{B_{\varrho}(x_0)}[H(x_0,1+|D\bfu|)]^{s_0}\,\mathrm{d}x  & \leq c \left(\dashint_{B_{2\varrho}(x_0)}H(x,1+|\E\bfu|)\,\mathrm{d}x\right)^{s_0} \\
& \leq c_2 (M_{x_0}+1)^{qs_0} \,.
\end{split}
\label{eq:highintgrad2}
\end{equation}
Assertion \eqref{eq:highintgradients} then follows combining \eqref{eq:highintgrad} and \eqref{eq:highintgrad2} and choosing $\tilde{c}:=\max\{c_1,c_2\}$.}

\noindent
{\bf Step 2:} From H\"older's inequality and the choice of $t$ in \eqref{eq:ok(5.25)} we get
\begin{equation*}
\begin{split}
\dashint_{B_\varrho(x_0)} H(x_0, |D\bfu-(D\bfu)_{x_0,\varrho}|)\,\mathrm{d}x & \leq \left(\dashint_{B_\varrho(x_0)} [H(x_0, |D\bfu-(D\bfu)_{x_0,\varrho}|)]^\frac{1}{q}\,\mathrm{d}x\right)^{tq} \\
& \,\,\,\, \times  \left(\dashint_{B_\varrho(x_0)} [H(x_0, |D\bfu-(D\bfu)_{x_0,\varrho}|)]^{s_0}\,\mathrm{d}x\right)^\frac{1-t}{s_0}\,.
\end{split}
\end{equation*}
Now, using Jensen's inequality for the concave function $\tilde{\Psi}$ such that $\frac{1}{2}\tilde{\Psi}(t)\leq \Psi(t) := [H(x_0,t)]^\frac{1}{q} \leq \tilde{\Psi}(t)$ (see Lemma~\ref{lem:lemma2.2ok}) and \eqref{eq:4.40ok},  we have
\begin{equation*}
\dashint_{B_\varrho(x_0)} [H(x_0, |D\bfu-(D\bfu)_{x_0,\varrho}|)]^\frac{1}{q}\,\mathrm{d}x \leq 2 [H(x_0,\varepsilon)]^\frac{1}{q} \leq 2[(1+\|\mu\|_\infty)\varepsilon^p]^\frac{1}{q}\,. 
\end{equation*}
On the other hand, using Jensen's inequality for the convex map $t\to [H(x_0,t)]^{s_0}$ we obtain
\begin{equation*}
\dashint_{B_\varrho(x_0)} [H(x_0, |D\bfu-(D\bfu)_{x_0,\varrho}|)]^{s_0}\,\mathrm{d}x \leq \dashint_{B_{\varrho}(x_0)}[H(x_0,1+|D\bfu|)]^{s_0}\,\mathrm{d}x\,.
\end{equation*}
Combining the previous estimates with \eqref{eq:highintgradients} we finally get
\begin{equation*}
\dashint_{B_\varrho(x_0)} H(x_0, |D\bfu-(D\bfu)_{x_0,\varrho}|)\,\mathrm{d}x \leq \bar{c}\varepsilon^{tp} (M_{x_0}+1)^{\frac{q(1-t)}{p}}\,,
\end{equation*}
for a constant $\bar{c}=\bar{c}(n,\nu,L,p,q,[\mu]_{C^\alpha},\|\mu\|_{L^\infty}, \|1+|\E\bfu|\|_{L^p})$. With the choice of $\varepsilon$ in \eqref{eq:(5.26)ok}, this implies
\begin{equation*}
\dashint_{B_\varrho(x_0)} H(x_0, |D\bfu-(D\bfu)_{x_0,\varrho}|)\,\mathrm{d}x \leq \frac{1}{c_*}\left(\frac{\varepsilon_*}{2}\right)^\frac{p}{2}\,. 
\end{equation*}
Now, using the change of shift formula with $\eta_*:=\frac{\varepsilon_*}{2^{n+1} (M_{x_0}+1) c_H}$ and $H(x_0,1)\geq1$, 
\begin{equation*}
\begin{split}
\frac{\Phi(x_0,\varrho)}{H(x_0, 1+|(\E\bfu)_{x_0,\varrho}|)} & \leq \frac{1}{H(x_0,1)}\dashint_{B_\varrho(x_0)}H_{1+|(\E\bfu)_{x_0,\varrho}|}(x_0, |\E\bfu - (\E\bfu)_{x_0,\varrho}|)\,\mathrm{d}x \\
& \leq \dashint_{B_\varrho(x_0)}H_{1+|(\E\bfu)_{x_0,\varrho}|}(x_0, |\E\bfu - (\E\bfu)_{x_0,\varrho}|)\,\mathrm{d}x \\
    & \leq c_* \dashint_{B_\varrho(x_0)}H(x_0, |\E\bfu - (\E\bfu)_{x_0,\varrho}|)\,\mathrm{d}x + \eta_* H_{1+|(\E\bfu)_{x_0,\varrho}|}(x_0, 1+|(\E\bfu)_{x_0,\varrho}|) \\
& \leq c_* \dashint_{B_\varrho(x_0)}H(x_0, |D\bfu - (D\bfu)_{x_0,\varrho}|)\,\mathrm{d}x + c_H \eta_* H(x_0, 1+|(\E\bfu)_{x_0,\varrho}|) \\
&  \leq \left(\frac{\varepsilon_*}{2}\right)^\frac{p}{2} + c_H \eta_* 2^n\dashint_{B_{2\varrho}(x_0)}H(x_0, 1+|\E\bfu|)\,\mathrm{d}x  \\
& \leq  \left(\frac{\varepsilon_*}{2}\right)^\frac{p}{2}  + \left(\frac{\varepsilon_*}{2}\right) \leq \varepsilon_* \,.
\end{split}
\end{equation*}
Moreover, with \eqref{eq:highintgradients} 
and \eqref{eq:smallvarrho} we have
\begin{equation*}
\begin{split}
\Theta(x_0,\varrho) & =  \varrho^\frac{1}{2} [H(x_0,\cdot)]^{-1}\left(\dashint_{B_{\varrho}}H(x_0,1+|D \bfu|)\,\mathrm{d}x\right) \\
     & \leq \varrho^\frac{1}{2} [H(x_0,\cdot)]^{-1}(\tilde{c}^\frac{1}{s_0}(M_{x_0}+1)^{q}) \\
 & \leq \delta_*\,.
\end{split}
\end{equation*}
By the absolute continuity of the integral, we can find an open neighborhood $U_{x_0}$ of $x_0$ such that
\begin{equation*}
\frac{\Phi(x,\varrho)}{H(x_0, 1+|(\E\bfu)_{x_0,\varrho}|)}<\varepsilon_*\quad \mbox{ and }\quad \Theta(x,{\varrho})<\delta_*
\end{equation*}
for every $x\in U_{x_0}$. Then, taking into account also \eqref{eq:5.24ok}, we can apply Lemma~\ref{lem:lemma3.13} at each point of $U_{x_0}$. This provides a Morrey-type estimate as in \eqref{(5.10Stroffo)} proving that $\bfu\in C^{0,\beta}(U_{x_0},\R^n)$ for every $\beta\in(0,1)$ (see, e.g., \cite[Lemma~3.15]{goodscistro}). Thus, $x_0\in\Omega_0$ and the proof is concluded.
\end{proof}

\section*{Acknowledgments }

The authors are members of Gruppo Nazionale per l'Analisi Matematica, la Probabilit\`a e le loro Applicazioni (GNAMPA) of INdAM.
G. Scilla and B. Stroffolini have been supported by the project STAR PLUS 2020 – Linea 1 (21‐UNINA‐EPIG‐172) ``New perspectives in the Variational modeling of Continuum Mechanics''.

\section*{Conflict of interest}
The authors  declare no conflict of interest.



\end{document}